\documentclass[12pt,a4paper]{article}

\usepackage[utf8]{inputenc}
\usepackage[OT1]{fontenc}
\usepackage{amsfonts}
\usepackage{amsmath}
\usepackage{amsthm}
\usepackage{amscd}
\usepackage{amssymb}
\usepackage{amsmath}
\usepackage{latexsym}
\usepackage{mathrsfs}
\usepackage{graphicx}
\usepackage{subfigure}
\usepackage{color}
\usepackage[english]{babel}
\usepackage{stmaryrd}
\usepackage{hyperref}
\usepackage{tabu}
\usepackage{booktabs}
\usepackage[body={16cm, 22cm},centering]{geometry}

\newtheorem{proposition}{Proposition}[section]
\newtheorem{lemma}{Lemma}[section]
\newtheorem{theorem}{Theorem}[section]
\newtheorem{corollary}{Corollary}[section]

\theoremstyle{definition}

\newtheorem{remark}{Remark}
\newtheorem{assumption}{Assumption}
\newtheorem{example}{Example}[section]

\newtheorem*{solution*}{Solution}

 \newcommand{\bbar}[1]{\setbox0=\hbox{$#1$}\dimen0=.2\ht0 \kern\dimen0 \overline{\kern-\dimen0 #1}}

 \DeclareMathOperator{\End}{\ensuremath{\mathcal{E}\kern-.125em\mathpzc{nd}}}

 \DeclareMathOperator{\Hom}{\mathcal{H}\kern-.125em\mathpzc{om}}
 
 \DeclareMathOperator{\id}{id}

 \newcommand{\udot}{\ensuremath{{\lower .183333em \hbox{\LARGE \kern -.05em$\cdot$}}}}

 \DeclareMathOperator{\Hess}	{Hess}

 \newcommand{\cL}{\mathcal{L}}

 \newcommand{\cU}{\mathcal{U}}

 \renewcommand{\P}{\mathbb{P}}
 
 \newcommand{\R}{\mathbb{R}}

  \newcommand{\Z}{\mathbb{Z}}

 \newcommand{\p}{\partial}
 \DeclareMathOperator{\gSO}{SO}
 \DeclareMathOperator{\gSp}{Sp}
 \DeclareMathOperator{\aSp}{sp}
 
 \DeclareMathOperator{\gU}{U}
 \DeclareMathOperator{\gGL}{GL}
 \DeclareMathOperator{\gO}{O}
 \DeclareMathOperator{\spn}{span}
 \DeclareMathOperator{\Mi}{Mi}

 \DeclareMathOperator{\sign}{sign}
 \DeclareMathOperator{\Mat}{Mat} 
 
\title{Jacobi Fields in Optimal Control: One-dimensional variations}
\author{A. Agrachev, I. Beschastnyi}

\begin{document}

\maketitle

\begin{abstract}
In this paper which is closely related to the previous paper [link] we specify general theory developed there. We study the structure of Jacobi fields in the case of an analytic system and piece-wise analytic control. Moreover, we consider only 1-dimensional control variations. Jacobi fields are piece-wise analytic in this case but may have jump discontinuities. We derive ODEs that these fields satisfy on the intervals of regularity and study behavior of the fields in a neighborhood of a singularity where the ODE becomes singular and the Jacobi fields may have jumps.
\end{abstract}

\section*{Introduction}

In this paper we continue the study of the second variation of optimal control problems using the technique of $\cL$-derivatives that were introduced in~\cite{agr_lderiv_1,agr_lderiv}. One can think of them as a rule that for a given critical point assigns to a certain space of variations a Lagrangian space in some symplectic space. In [link] we presented the theoretical basis and gave an algorithm how to compute an approximation of an $\cL$-derivative with arbitrary good precision. 

In this article we focus on examples and simpler characterization of Jacobi curves. To be more precise we study the following optimal control problem
\begin{equation}
\label{eq:control}
\dot{q} = f(q,u), \qquad u \in U \subset \R, \qquad q\in M
\end{equation}
$$
q(0) = q_0, \qquad q(T) = q_T,
$$
$$
J_T[u] = \int_0^T L(u,q)dt \to \min.
$$
Here $M$ is a $n$-dimensional manifold. For simplicity we assume that $f(u,q)$ and $L(u,q)$ are analytic in both variables, that $U$ is a polytope or a one-dimensional smooth manifold, that controls are $L^\infty$ functions and that the time $T$ is fixed. For simplicity and conceptual clarity we make the assumption that extremal control takes values in the vertices or one-dimensional edges of $U$.

We begin by recalling the necessary results from symplectic geometry in Section~\ref{sec:sympl}. This part is standard and can be found in several books like~\cite{gosson} or~\cite{sternberg}. We present this section for reader's convenience and in order to fix the basic notations. Then we recall all the necessary results from our article~\cite{A_and_me} about the $\cL$-derivatives and the construction of  Jacobi curves for optimal control problems in Sections~\ref{sec:l-der} and~\ref{sec:statement}.

Using these techniques we then proceed to the study of singular and bang-bang cases in Sections~\ref{sec:jacobide} and~\ref{sec:bang}, where we give efficient algorithms for characterizing Jacobi curves. These cases were already studied by many authors (see for example~\cite{bonnard_book,sussmann,SussmannLiu,osmolovskii,goc,agrachev_bang}). These results can be recovered using the constructions from this paper and Morse-type theorems from our previous article~\cite{A_and_me}.

In the second part we study extremals along which at a single point the Legendre condition becomes degenerate. In this case the Jacobi curve should be a solution of a singular Jacobi DE, but it can not be characterized as the usual boundary value problem, because we lose both existence and uniqueness of solutions. This kind of systems were previously studied in the classical calculus of variations, in particular by Morse himself and some of his students~\cite{morse2,morse3} using functional-analytic techniques or in~\cite{vagner} using differential geometry and fields of extremals. Some special examples with similar singularities were recently studied in~\cite{caillau}. We show how using the technique of $\cL$-derivatives we can still characterize the desired Jacobi curve in a simplest singular example as boundary value problem of an ODE, but with conditions on the first k-jet of a solution (and not just the initial value). 

In Section~\ref{sec:normal_form} we show that if we restrict ourselves only to one-dimensional variations, using a change of variables under some non-degeneracy conditions we can separate the dynamics of the system into a regular and singular part in an invariant symplectic subspace of a dimension at most four. In Section~\ref{sec:easy_case} we explain the idea heuristically when the dimension of the manifold on which the control system is defined is equal to one. In the remaining sections we make all those ideas rigorous. In Section~\ref{sec:kneser}, we find sufficient conditions for existence and non-existence of $\cL$-derivatives and in Section~\ref{sec:jumpg2} we characterize Jacobi curves as singular BVPs.

\section{Linear symplectic geometry and Lagrangian Grassmanian}
\label{sec:sympl}
In this section we recall basic facts from symplectic geometry and fix notations that we use in the article.

Given a symplectic space $(\R^{2n},\sigma)$, where $\sigma$ is a symplectic form we can always assume by the Darboux theorem that
$$
\sigma(\lambda_1,\lambda_2) = \lambda_1^T J \lambda_2, \qquad \forall \lambda_i \in \R^{2n},
$$  
where $J$ is the standard complex structure
$$
J = \begin{pmatrix}
0 & \id_n\\
-\id_n & 0
\end{pmatrix}.
$$
In particular $J^2 = -\id_{2n}$. Coordinates in which $\sigma$ has such a form are called \textit{Darboux coordinates}. We use usual position-momenta notations in this case, i.e. we write $\lambda=(p,q)=(p_1,...,p_n,q_1,...,q_n)$.

In Darboux coordinates a Hamiltonian system with a perhaps time-dependent Hamiltonian $H: \R^{2n}\times \R \to \R$ is a system of ODEs
$$
\dot{\lambda} = - J\nabla H(t,\lambda),
$$
where $\nabla H$ is the $\R^{2n}$-gradient of $H$. In particular, if $H$ is quadratic of the form
$$
H(t,\lambda) = \frac{p^T C(t) p - 2q^T A(t) p - q^T B(t) q}{2},
$$
where $B(t), C(t)$ are symmetric matrices, we obtain a linear Hamiltonian system
\begin{equation}
\label{eqintro_ham}
\frac{d}{dt}\begin{pmatrix}
p\\
q
\end{pmatrix} 
=
\begin{pmatrix}
A(t) & B(t)\\
C(t) & - A^T(t)
\end{pmatrix}
\begin{pmatrix}
p\\
q
\end{pmatrix}.
\end{equation}

Given $J$ we can define the \textit{symplectic group} $\gSp(2n)$ and the corresponding \textit{symplectic algebra} $\aSp(2n)$ as
$$
\gSp(2n) = \left\{M\in \Mat(2n\times 2n,\R)\,:\, M^TJ M = J\right\},
$$
$$
\aSp(2n) = \left\{X\in \Mat(2n\times 2n,\R)\,:\, X^TJ  + JX = 0\right\}.
$$
If we write down $X\in \aSp(2n)$ as block matrix, we will see that it has the same form as the matrix in the Hamiltonian system \eqref{eqintro_ham}. Therefore we immediately can see that the flow $\Phi(t)$ of \eqref{eqintro_ham} is symplectic.

We define the \textit{skew-orthogonal complement} of a linear subspace $\Gamma$ as 
$$
\Gamma^\angle = \{\mu \in \R^{2n} \, : \,\sigma(\mu,\lambda) = 0, \, \forall \lambda \in \Gamma \}.
$$

A subspace $\Gamma$ is called
\begin{itemize}
\item \textit{isotropic} if $\Gamma \subset \Gamma^\angle $ or equivalently if $\sigma|_{\Gamma} = 0$,
\item \textit{Lagrangian} if $\Gamma = \Gamma^\angle $ or if equivalently $\sigma|_{\Gamma} = 0$ and $\dim \Gamma = n$,
\item \textit{coisotropic} if $\Gamma \supset \Gamma^\angle $,
\item \textit{symplectic} if $\dim (\Gamma \cap \Gamma^\angle) = 0$ or if equivalently $\sigma|_\Gamma$ is non-degenerate.
\end{itemize}

Since $\sigma$ is skew-symmetric, any one-dimensional direction $\R v$, $v\in\R^{2n}$ is isotropic.  Two main examples of Lagrangian subspaces are the \textit{horizontal subspace} $\Sigma$ and the \textit{vertical subspace} $\Pi$ defined as
\begin{align*}
\Pi &= \left\{(p,q)\in \R^{2n} \, : \, q=0  \right\},\\
\Sigma &= \left\{(p,q)\in \R^{2n} \, : \, p=0  \right\}.
\end{align*}
We can construct other examples as follows. Let $S = S^T$ be a symmetric matrix. Then 
$$
\Lambda_S = \{(p,Sp)\in \R^{2n}\,:\, p\in \R^n\}
$$
is a Lagrangian subspace transversal to $\Sigma$. Conversely to any and Lagrangian subspace $\Lambda$ transversal to $\Sigma$ (we denote this by $\Lambda \pitchfork \Sigma$) we can associate a symmetric operator $S$ from $\Pi$ to $\Sigma$.

We call the set of all Lagrangian planes \textit{Lagrangian Grassmanian} and denote it by $L(n)$. It is a manifold, whose atlas is given by $\Lambda^\pitchfork$, which are the sets of Lagrangian planes transversal to $\Lambda \in L(n)$. Coordinate charts are maps from $\Lambda^\pitchfork$ to the space of symmetric matrices constructed like above. Throughout this paper we use another representation of a Lagrangian plane $\Lambda \in L(n)$ as a span of $n$ independent vectors $v_i$. It is clear that such a representation is not unique. We can replace $v_i$ by any linear span of the same vectors as long as they remain independent. This means that in general we need to quotient a natural $\gGL(n)$ action. We can arrange $v_i$ in a single $n\times 2n$ matrix and we  write
$$
\Lambda = \begin{bmatrix}
v_1 & ... & v_n
\end{bmatrix},
$$
where the square brackets indicate the equivalence class under the $\gGL(n)$ action. We denote this action by
$$
g \begin{bmatrix}
v_1 & ... & v_n
\end{bmatrix} := 
\begin{bmatrix}
 gv_1 & ... & gv_n
\end{bmatrix}, \qquad g\in \gGL(n). 
$$

For example, if $\Lambda\in \Sigma^\pitchfork$ we can write
$$
\Lambda = \begin{bmatrix}
\id_{n}\\
S
\end{bmatrix},
$$
where $S$ is a symmetric matrix like in the example above. Or we can assume that $v_i$ form an orthonormal basis of $\Lambda$ in $\R^{2n}$. Then
$$
\Lambda = \begin{bmatrix}
X\\
Y
\end{bmatrix},
$$
where $X^TX + Y^T Y = \id_{n}$ (orthonormality property) and $X^TY - Y^T X = 0$ (Lagrangian property) are satisfied. A matrix $X+iY$ that satisfies these properties is unitary and the converse is true as well. We can choose $v_i$ in such a way up to a $\gO(n)$-action, which is given by
$$
\begin{pmatrix}
O & 0 \\
0 & O
\end{pmatrix}
\begin{bmatrix}
X\\
Y
\end{bmatrix}, \qquad O \in \gO(n).
$$
This gives the usual identification of $L(n) \simeq \gU(n)/\gO(n)$.

We will use this idea many times when we will consider the singular case, so at this point it makes sense to consider a simple example that will be useful for us later. 

\begin{example}
\label{ex:1}
Suppose that we would like to find a simple representation of a Lagrangian plane $\Lambda\in L(2)$ knowing that $\dim(\Lambda\cap\Sigma) = 1$. Then it must be of the form
$$
\Lambda = \begin{bmatrix}
v_1 & v_2
\end{bmatrix} = 
\begin{bmatrix}
x_1 & 0 \\
y_1 & 0 \\
z_1 & z_2 \\
w_1 & w_2 
\end{bmatrix}, \qquad x_1^2 + y_1^2 \neq 0.
$$
We can assume that $v_1$ and $v_2$ are orthonormal. We then apply a rotation $O \in \gO(n)$, so that $y$ component of $v_1$ becomes zero. Then
$$
\Lambda =\begin{bmatrix}
Ov_1 & Ov_2
\end{bmatrix} = 
\begin{bmatrix}
\sqrt{x_1^2 + y_1^2} & 0 \\
0 & 0 \\
\tilde z_1 & \tilde z_2 \\
\tilde w_1 & \tilde w_2 
\end{bmatrix},
$$
but since $\Lambda$ is Lagrangian we must have $\tilde z_2 = 0$. Changing the basis we then find
$$
\Lambda = 
\begin{bmatrix}
\dfrac{Ov_1 - (\tilde{w}_1/\tilde{w}_2)Ov_2}{\sqrt{x_1^2 + y_1^2}} & \dfrac{Ov_2 }{\tilde{w}_2}
\end{bmatrix}
=
\begin{bmatrix}
1 & 0\\ 
0 & 0\\
z & 0\\
0 & 1
\end{bmatrix}.
$$
\end{example}

Given an isotropic subspace $\Gamma$ and a Lagrangian plane $\Lambda$, we can construct a new Lagrangian plane $\Lambda^\Gamma$, which is a Lagrangian plane that contains $\Gamma$ and the dimension of $\Lambda \cap \Lambda^\Gamma$ is maximal. It is defined as
$$
\Lambda^\Gamma = (\Lambda \cap \Gamma^\angle) + \Gamma = (\Lambda+\Gamma) \cap \Gamma^\angle.
$$
If $\Gamma = \R X$ for some vector $X\in\R^{2n}$ we will simply write $\Lambda^X$ instead of $\Lambda^{\R X}$.

Let us have a look at another example that will be useful in future. 
\begin{example}
\label{ex2}
Assume that $X = \begin{pmatrix}1 & 0 & 0 & 0\end{pmatrix}^T$ and we would like to construct $\Lambda^{X}$ for $\Lambda \in L(2)$. We have that
$$
\Lambda = \begin{bmatrix}
v_1 & v_2
\end{bmatrix} =
\begin{bmatrix}
x_1 & x_2 \\
y_1 & y_2 \\
z_1 & z_2 \\
w_1 & w_2
\end{bmatrix}.
$$
Subspace $X^\angle$ consists of vectors $v \in \R^4$ whose $z$-component is zero. So assume first that $z_1 = z_2 = 0$. Then $\sigma(X,v_1) = \sigma (X,v_2) = 0$. But since it is a Lagrangian subspace, it means that $X\in \Lambda$ and by definition $\Lambda^X = \Lambda$. Thus we can take $v_1 = X$. In this case we obtain
$$
\Lambda^X = 
\begin{bmatrix}
X & v_2
\end{bmatrix} =
\begin{bmatrix}
X & v_2 - x_2X
\end{bmatrix} =
\begin{bmatrix}
1 & 0 \\
0 & y_2 \\
0 & 0 \\
0 & w_2
\end{bmatrix}.
$$

Suppose that $X\notin \Lambda$. Then $z_1^2 + z_2^2 \neq 0$ and as a result $\sigma(X, z_1v_1 + z_2v_2) \neq 0$, but $\sigma(X, z_1v_2 - z_2v_1) = 0$. So $\Lambda \cap X^\angle = z_1v_2 - z_2v_1$ and by definition
$$
\Lambda^X = 
\begin{bmatrix}
X & z_1v_2 - z_2v_1
\end{bmatrix} =
\begin{bmatrix}
1 & z_1x_2 - z_2x_1 \\
0 & z_1y_2 - z_2y_1 \\
0 & 0 \\
0 & z_1w_2 - z_2w_1
\end{bmatrix} =
\begin{bmatrix}
1 & 0 \\
0 & z_1y_2 - z_2y_1 \\
0 & 0 \\
0 & z_1w_2 - z_2w_1
\end{bmatrix}.
$$
\end{example}

Our main objects of study are going to be curves in the Lagrangian Grassmanian. The curves will always come from a flow $\Phi(t)$ of a linear Hamiltonian system \eqref{eqintro_ham}. We will simply take a point $\Lambda$ and consider a curve $\Lambda(t) = \Phi(t)\Lambda$. More generally a linear Hamiltonian system induces a dynamical system on $L(n)$. We can write down an ODE for that system using local charts. Indeed, let $(p(t),q(t))$ be a solution of \eqref{eqintro_ham} and $S(t)$ be a curve of symmetric matrices that correspond to $\Lambda(t)$. Then $q(t) = S(t)p(t)$ and we differentiate this expression. This way we obtain a Riccati equation of the form 
\begin{equation}
\dot{S} + SA+A^TS +SBS - C =0. \label{eqriccati_general}
\end{equation}
Since a coordinate chart $\Sigma^\pitchfork$ is dense in $L(n)$ the opposite is also true: a Riccati equation of the form above gives rise to a Hamiltonian system and a well defined flow on $L(n)$. In order to write down a Riccati equation in a different chart we can apply a symplectic transformation to the corresponding Hamiltonian, s.t. a given Lagrangian plane $\Lambda$ is mapped to $\Sigma$. Then we simply insert the new expressions for $A,B$ and $C$ in \eqref{eqriccati_general}.

The last ingredient that we need is the \textit{Maslov index}, which is a topological symplectic invariant of curves in the Lagrangian Grassmanian. Maslov index is the same for different curves in the same homotopy class, thus it does not change under small perturbations. We are going to give a simple coordinate definition, but one should remember that this number has many equivalent invariant definitions (see for example~\cite{maslov}). 

Let $\Lambda:[0,1] \to L(n)$ be a continuous curve. We assume that the curve $\Lambda(t)$ lies completely in some coordinate chart, i.e. there exists a plane $\Delta\in L(n)$, s.t. $\Lambda(t) \in \Delta^\pitchfork$ for all $t\in[0,1]$. We call such a curve \textit{simple}. Let $\Pi\in \Delta^\pitchfork$, and assume that end-points $\Lambda(0)$ and $\Lambda(1)$ are transversal to $\Pi$. We can then associate to $\Lambda(0),\Lambda(1)$ symmetric matrices $S_0,S_1$ that correspond to symmetric operators from $\Pi$ to $\Delta$. Maslov index of $\Lambda(t)$ with respect to $\Pi$ is defined as
$$
\Mi_\Pi \Lambda(t) = \frac{1}{2}\left( \sign S_1 - \sign S_0 \right).
$$
One can check that this definition does not depend on the choice of $\Delta$. To define the Maslov index for a general curve $\Lambda:[0,1] \to L(n)$ one should split it in a number of simple arcs $\Lambda_i(t)$, s.t. each arc lies in its own coordinate chart. Then the Maslov index of the whole curve is defined as a sum of the corresponding indices of simple arcs:
$$
\Mi_\Pi \Lambda(t) = \sum_{i=1}^N \Mi_{\Pi} \Lambda_i(t).
$$
From this definition we can easily see that the index of a simple curve depends only on the relative position of its end-points and remains the same if we perturb the curve as long as the transversality conditions are preserved. Using this one can prove that Maslov index of a curve is a homotopy invariant.

We have two very important properties of the Maslov index related to  a change of the reference plane.
\begin{theorem}[\cite{agr_gamk_symp}]
If a curve $\Lambda(\tau)\subset L(n)$ is closed, then its Maslov index does not depend on the choice of the reference plane, i.e.
$$
\Mi_{\Delta_1} \Lambda(\tau) = \Mi_{\Delta_2} \Lambda(\tau) = \Mi \Lambda(\tau), \qquad \forall \Delta_i \in L(n).
$$
If it is not closed we have the following estimate
$$
|\Mi_{\Delta_1} \Lambda(\tau) - \Mi_{\Delta_2} \Lambda(\tau)| \leq n.
$$
\end{theorem}

There exist several other equivalent definitions of the Maslov index~\cite{gosson,maslov}. Usually one defines it as an intersection index, but this definition is rather long and has many subtleties. Nevertheless it allows to prove easily the following theorem that we will use in Section~\ref{sec:kneser}.

\begin{theorem}[\cite{agr_gamk_symp}]
\label{thm:comparison}
Let $H(t)$ be a quadratic non-autonomous Hamiltonian and let $\Phi(t)$ be a flow of the corresponding Hamiltonian system. Fix two transversal Lagrangian planes $\Delta$ and $\Lambda$ and assume that $\Delta \pitchfork \Phi(T)\Lambda$. If $H(t)|_{\Delta} \geq 0$ for $t\in[0,T]$, then
$$
\Mi_{\Delta} \Phi(t) \Lambda = \sum_{t\in[0,T]} \dim\left(\Delta \cap \Phi(t)\Lambda  \right).
$$
If $H(t)|_{\Delta} \leq 0$, then the same formula holds with a minus sign in front of the sum.
\end{theorem}

These two theorems will be our main tool in proving oscillation results in Section~\ref{sec:kneser}. 

\section{$\cL$-derivatives}
\label{sec:l-der}

We will derive the Jacobi equations and study them using the so called $\cL$-derivatives. A $\cL$-derivative is a rule that assigns to an admissible space of variations a Lagrangian plane in some symplectic space. As we add variations we can compare the relative positions of the corresponding $\cL$-derivatives and deduce from that how the inertia indices and nullity of the Hessian change as we consider a bigger and bigger space of variations. As a result one can recover the classical theory of Jacobi and much more. This theory is applicable in a great variety of cases, even when there is no Jacobi equation at all. We begin by explaining the abstract setting and then we specialize the results to optimal control problems.

Assume that we have the following constrained variational problem. Let $J: \cU \to \R$ be a smooth functional and $F: \cU \to M$ be a smooth map, where $\cU$ is a Banach manifold and $M$ is a finite-dimensional manifold. Given a point $q\in M$, we are interested in finding $\tilde{u} \in F^{-1}(q)$ that minimize $J$ among all other points $u\in F^{-1}(q) $. In the case of optimal control problems $\cU$ is the space of admissible controls. The map $F$ is usually taken to be the end-point map, which we will introduce in the next section. 

The first step is to apply the Lagrange multiplier rule that says that if $\tilde{u}$ is a minimal point then there exists a covector $\lambda \in T_q^* M$ and a number $\nu\in\{0,1\}$, s.t. 
\begin{equation}
\label{eqlagrange}
\langle \lambda, dF[\tilde u](w) \rangle - \nu dJ[\tilde u](w) = 0,\qquad \forall w\in T_{\tilde{u}}\cU. 
\end{equation}
A pair $(\tilde{u},\lambda)$ that satisfies the equation above is called a \textit{Lagrangian point} and $\tilde{u}$ is called a critical point of $(F,J)$. There are of course many critical points that are not minimal. So in order to find the minimal ones we have to apply higher order conditions for minimality. For example, we can look at the Hessian $\Hess(F,\nu J)[\tilde{u},\lambda]$ at a Lagrangian point $(\tilde{u},\lambda)$ that we define as
\begin{equation}
\label{eqhessian}
\Hess (F,\nu J)[\tilde u, \lambda] := \left( \nu d^2 J[\tilde u] - \langle \lambda, d^2 F[\tilde{u}] \rangle \right)|_{\ker dF[\tilde{u}]}.
\end{equation}
In the normal case this expression coincides with the Hessian of $J$ restricted to the level set $F^{-1}(q)$. The index and the nullity of the Hessian are directly related to optimality of the critical point $\tilde{u}$~\cite{as}. 

We are now ready to define $\cL$-derivatives. We linearise \eqref{eqlagrange} with respect to $\lambda$ and $u$, and obtain the following equation
$$
\langle \xi, dF[\tilde u](w) \rangle + \langle \lambda, d^2F[\tilde u](v,w) \rangle  - \nu d^2J[\tilde u](v,w)=0.
$$
Or if we define $Q(v,w) := \langle \lambda, d^2F[\tilde u ](v,w) \rangle + \nu d^2J[\tilde{u} ](v,w)$, we can rewrite this as
\begin{equation}
\label{eql_deriv_def}
\langle \xi, dF[\tilde u](w) \rangle + Q(v,w)=0.
\end{equation}

A \textit{$\cL$-derivative} of a pair $(F,J)$ at a Lagrangian point $(\tilde{u},\lambda)$ constructed over a finite-dimensional space of variations $V \subset T_{\tilde{u}} \cU$ is the set
$$
\cL(F,J)[\tilde{u},\lambda](V) = \{(\xi,dF[\tilde{u}](v))\in T_{\lambda}(T^*M) \, : \, (\xi,v)\in (T_{\lambda}(T^*_{q}M), V) \text{ solve \eqref{eql_deriv_def} for } \forall w\in V\}.
$$
This set is a Lagrangian plane~\cite{agrachev_cime}. The reason why we do not take directly $T_{\tilde{u}} \cU$ instead of $V$ is that it is a linear equation defined on an infinite-dimensional space and it might be ill-posed. In this case $\cL(F,J)[\tilde{u},\lambda](V)$ is just isotropic. But if we have chosen the right topology for our space of variations, we are going to get exactly $\dim M$ independent solutions. 

To define a $\cL$-derivative over an infinite-dimensional space $V \subset T_{\tilde{u}} \cU$ we must take a generalized limit or a limit of a net over all finite dimensional subspaces $U\subset V$:
$$
\cL(F,J)[\tilde{u},\lambda](V) = \lim_{U\nnearrow V} \cL(F,J)[\tilde{u},\lambda](U),
$$
with the partial ordering given by inclusion. When $V$ is the whole space of available variations, we simply write $\cL(F,J)[\tilde{u},\lambda]$ for the corresponding $\cL$-derivative. 

We have the following important theorem proved in~\cite{agr_lderiv}, that gives the existence of this limit and a way to compute it.
\begin{theorem}
\label{thm:main}
Let $(\tilde{u},\lambda)$ be a Lagrangian point of $(F,J)$. 
\begin{enumerate}
\item If either the positive or the negative inertia index of $\Hess (F,\nu J)[\tilde u,\lambda]$ is finite, then $\cL(F,J)[\tilde{u},\lambda]$ exists;
\item $\cL(F,J)[\tilde{u},\lambda] = \cL(F,J)[\tilde{u},\lambda](V)$ for any $V$ dense in $T_{\tilde{u}}\cU$.
\end{enumerate}
\end{theorem}

$\cL$-derivatives contain information about the inertia indices and nullity of the Hessian \eqref{eqhessian} restricted to some space of variations. By comparing two $\cL$-derivatives constructed over two subspace $V\subset W$, we can see how the inertia indices change as we add variations to our variations space~\cite{agr_lderiv}. 

In the next section we will write down explicit expressions for the optimal control problem we are studying, and define Jacobi curves. In article~\cite{A_and_me} we gave a general algorithm for their approximation. In this article we focus on situations when an exact characterization is possible.

\section{$\cL$-derivatives for optimal control systems}
\label{sec:statement}
Let us consider the optimal control problem~\eqref{eq:control}. We introduce another type of variations called time variations. They can be used to obtain necessary~\cite{agrachev_bang2} and sufficient~\cite{agrachev_bang} optimality conditions. We have discussed them in detail in~\cite{agrachev_moi} (see also~\cite{A_and_me}), so we just give a quick recap. 

Let us consider a change of time 
$$
t(s) = \int_0^s (1+\alpha(\theta)) d\theta,
$$
where we assume $\alpha(\theta) > -1$, in order to have invertibility of $t(s)$. Then in the new time we can write
\begin{align*}
\dot{q} = (1+\alpha(s)) f(q,u(t(s))), \\
\dot{t} = 1+\alpha(s),
\end{align*}
$$
\int_{0}^{t^{-1}(T)} (1+\alpha(s))L(q,u(t(s)))ds \to \min,
$$
where we have added the time as a new variable in order not to add any additional constraints on the controls. One can then proceed to show that $\tilde{u}(t)$ is minimal if and only if the control $(\alpha(s),u(s)) = (0,\tilde{u}(t(s))$ is minimal in the new problem above. We will denote $\hat U = U\times (-1,\infty)$ the new space of control values and $\hat{u}$ a critical control in $\hat U$. 
\begin{remark}
By abusing slightly the notations we consider the pair $(u,\alpha)$ consisting of the original control parameter $u$ and time variations $\alpha$ as a new control, that we denote again by $u$.
\end{remark}
 
In order to reformulate the optimal control problem above as a constrained variational problem of the previous section one  introduces the \textit{end-point map} $E_t : L^\infty([0,T],\hat U) \to M$. It takes an admissible control $u(t)$ and associates to it final point of the trajectory which is a solution of the Cauchy problem \eqref{eq:control} with $q(0) = q_0$. Then in the notations of the previous section $J=J_T$, $E=E_T$. Applying the Lagrange multiplier rule we find that if $(\tilde{u}(t),\tilde{q}(t))$ is an optimal control and the corresponding optimal trajectory, then there must exist a covector $\lambda(T) \in T^*_{\tilde{q}(T)} M$ and a number $\nu\in\{0,1\}$, s.t.
$$
\langle \lambda(T), dE_T[\tilde{u}](w)\rangle - \nu dJ_T [\tilde{u}](w) = 0,
$$
for suitable variations $w$. Let $P_{t}^T$ be the flow of \eqref{eq:control} under the control $\tilde u(t)$ from time $t$ to time $T$. From the definitions it is clear that $dE_T|_{L^\infty([0,t],\hat U)} = (P^T_t)_* dE_t$ and $dJ_T|_{L^\infty([0,t],\hat U)} = dJ_t$, where $(P^T_t)_*$ is the differential of $P^T_t$. If we denote $\lambda(t) = (P^T_t)^* \lambda(T) \in T^*_{\tilde{q}(t)}M$ then, we have along $\tilde{q}(t)$
\begin{equation*}
\langle \lambda(t), dE_t[\tilde{u}](w)\rangle - \nu dJ_t [\tilde{u}](w) = 0.
\end{equation*}
These first order conditions are equivalent to a weak version of the Pontryagin maximum principle which states that $\lambda(t)$ must satisfy a Hamiltonian system
\begin{equation}
\label{eq:ham}
\dot \lambda = \vec{h}_{\tilde{u}(t)}\lambda,
\end{equation}
where 
$$
h(u,\lambda) = \langle \lambda, f(u,q)\rangle - \nu L(u,q),
$$
and along the extremal curve an extremum condition
$$
\left.\frac{\p h(u,\lambda(t))}{\p u}\right|_{u = \tilde{u}} = 0
$$
is satisfied.

From the previous discussion it would seem natural to compute $\cL$-derivatives of $(E_t,J_t)$ in order to study the second variation. In this case we would obtain a one-parametric family of Lagrangian planes which encode information about the corresponding Hessian. However in this case Lagrangian planes will lie in different symplectic spaces and thus we can not compare directly their relative position. To fix this problem, we simply take the flow of the Hamiltonian system~\eqref{eq:ham} that we denote as $\Phi_t$ and compute $\Phi_t^* \cL(E_t,J_t)[\tilde{u},\lambda](V)$. 

\begin{remark}
From now on we will consider $\Phi_t^* \cL(E_t,J_t)[\tilde{u},\lambda](V)$ only and for brevity we will omit the pull-back by just writing $\cL(E_t,J_t)[\tilde{u},\lambda](V)$.
\end{remark}

It remains to give an explicit expression of \eqref{eql_deriv_def}.
We define
$$
b(\tau) =\left. \frac{\p^2}{\p u^2}\right|_{u=\tilde{u}} h(u,\lambda(t))
$$
and
$$
X(\tau) = \left. \frac{\p}{\p u}\right|_{u=\tilde{u}} \overrightarrow{\Phi_t^*(h(u,\cdot))(\lambda(0))},
$$
which is essentially the linearization of a pull-back of the PMP Hamiltonian using the flow $\Phi_t$. In~\cite{A_and_me} we have proven the following characterization of the previously mentioned $\cL$-derivatives of optimal control problems.

\begin{proposition}
\label{prop:l_def}
$\cL(E_t,J_t)[\tilde{u},\lambda]$ consists of the vectors of the form
\begin{equation}
\label{eql_der_def1}
\eta_t = \eta_0 + \int_0^t X(\tau) v(\tau) d\tau,
\end{equation}
where $\eta_0 \in T_{\lambda_0}(T^*_{q_0}M)$ and $v\in V$ satisfying
\begin{equation}
\label{eql_der_def2}
\int_0^t \left( \sigma\left( \eta_0 + \int_0^\tau X(\theta) v(\theta) d\theta, X(\tau)w(\tau)  \right) + b(\tau)(v(\tau),w(\tau)) \right)d\tau =0, \qquad \forall w\in V.
\end{equation}
\end{proposition}

One can use this proposition directly as a working definition. It immediately allows to prove an important property of $\cL$-derivatives of optimal control problems.

\begin{lemma}
\label{lemm:add}
Let $0<t_1<t_2$ and assume that $\cL(E_{t_2},J_{t_2})[\tilde{u},\lambda]$ exists. We denote by $V$ some finite-dimensional subspace of $L^\infty([t_1,t_2],\hat U)$ and we consider the following equation
\begin{equation}
\label{eqalt_l_der}
\int_{t_1}^{t_2} \left( \sigma\left( \lambda + \int_0^\tau X(\theta) v(\theta) d\theta, X(\tau)w(\tau)  \right) + b(\tau)(v(\tau),w(\tau)) \right)d\tau =0, \qquad \forall w\in V,
\end{equation}
where $v \in V$ and $\lambda \in \cL(E_{t_1},J_{t_1})[\tilde{u},\lambda]$. Then we can characterize alternatively $\cL(E_{t_2},J_{t_2})[\tilde{u},\lambda]$ as a generalized limit of Lagrangian subspaces
$$
\left\{ \lambda + \int_{t_1}^{t_2} X(\tau)v(\tau) d\tau \, : \, \lambda \in \cL(t_1), v\in V \text{ satisfy \eqref{eqalt_l_der} for any } w \in V    \right\}.
$$
\end{lemma}
This lemma implies that we can construct $\cL(E_{t},J_{t})[\tilde{u},\lambda]$ by knowing already the same $\cL$-derivative at time $\tau < t$. This is obvious when we have a Jacobi DE, since in this case $\cL(E_{t},J_{t})[\tilde{u},\lambda]$ can be defined using the flow of Jacobi equation.

To include maximum information about the Hessian one should take $L^\infty([0,T],\hat{U})$ as the space of admissible variations, but for simplicity and conceptional clarity we will take a smaller space $\cU \subset L^\infty([0,T],\hat{U})$ defined in the following way. If on some interval $\tilde{u}(\tau)$ takes values on the boundary $\p U$, then we only use time variations. Otherwise we use variations of the control $u$. Therefore we always use \textit{only} one-dimensional variations. We define Jacobi curves as $\cL_t = \cL(E_{t},J_{t})[\tilde{u},\lambda](\cU\cap L^\infty([0,t],\hat{U}))$. 

In this paper we make the following assumption on the regularity of our system and extremal controls:

\begin{assumption}
Functions $b(t)$, $X(t)$ and the extremal control $\tilde{u}(t)$ are piece-wise analytic as functions of $t$.
\end{assumption}
\begin{assumption}
We assume that the extremal control $\tilde{u}(t)$ takes values either on a vertex or on an edge of the polytope $U$.
\end{assumption}
\begin{assumption}
When extremal control $\tilde{u}(t)$ is on a vertex, we only use time variations, when $\tilde{u}(t)$ is on a an edge, we only use variations along that edge.
\end{assumption}

\begin{remark}

These assumptions are not as restrictives as it might seem. Time variations have no effect if the extremal control $\tilde{u}(\tau)$ is smooth, because in this situation any time variation $\alpha$ can be realised as a variation of the control parameter $u$~\cite{agrachev_moi}. Therefore the effect of time variations is concentrated at discontinuities of the reference control $\tilde{u}(\tau)$. Under the piecewise analyticity assumption we can have discontinuities only at isolated points. The constructed Jacobi curve will be weaker then the one constructed using all possible variations. But even in this case we can obtain useful optimality conditions, and it is possible to generalize this argument to other situations by simply considering any subspace of one-dimensional two-sided variations.  
\end{remark}

In~\cite{A_and_me} by using Lemma~\ref{lemm:add} we gave an iterative algorithm for the construction of the Jacobi curve $\cL_t$. The idea was to split the time interval $[0,T]$ into small pieces and on each piece to use the space of constant functions as $V$. This way one obtains an approximation of the Jacobi curve which by Theorem~\ref{thm:main} converges point-wise to it, when the length of the biggest interval of the splitting goes to zero. This algorithm becomes particularly nice if we have a single control parameter, like in the problem that we consider in this article. In this case $b(\tau)$ is just a function and $X(\tau)$ is a $\R^{2n}$-valued vector function.

\begin{proposition}
\label{prop:sungle_u}
Consider a single control parameter system. Given a $\cL$-derivative $\cL_t(V)$, where $V$ is some space of variations defined on $[0,t]$, we have $\cL_t(V \oplus \R \chi_{[t,t+\varepsilon]}) = \cL_t(V)^{\eta(t+\varepsilon)}$, where $\chi_{[t,t+\varepsilon]}$ is the characteristic function of the corresponding interval and $\eta(t+\varepsilon)$ is determined by one of the two alternatives
\begin{enumerate}
\item If 
\begin{equation*}
\int_t^{t+\varepsilon}X(\tau) d\tau \in \cL_t(V)
\end{equation*}
then $\cL_t(V \oplus \R \chi_{[t,t+\varepsilon]}) = \cL_t(V)$ and we can take $\eta(t+\varepsilon)$ to be any vector from $\cL_t(V)$

\item Else we fix any $\eta(t) \in \cL_t(V)$ satisfying
$$
\sigma\left( \eta(t), \int_t^{t+\varepsilon} X(\tau) d\tau \right) \neq 0
$$
and take
$$
\eta(t+\varepsilon) = K \eta(t) + \frac{1}{\varepsilon}\int_t^{t+\varepsilon} X(\tau) d\tau
$$
where 
$$
K = -\frac{\dfrac{1}{\varepsilon}\displaystyle\int_t^{t+\varepsilon}\left[\sigma\left( \int_t^\tau X(\theta) d\theta, X(\tau) \right)+b(\tau) \right]d\tau}{\sigma\left( \eta(t), \displaystyle\int_t^{t+\varepsilon} X(\tau) d\tau \right) }.
$$
\end{enumerate}
\end{proposition}

In~\cite{A_and_me} we have also proven the following useful lemma.
\begin{lemma}
\label{lem:continuity}
The Jacobi curve $\cL_t$ is left continuous.
\end{lemma}

At the end of this section we briefly summarize the key-points regarding $\cL$-derivatives.
\begin{enumerate}
\item A $\cL$-derivative is a map that assigns to a critical point of the functional and a subspace of admissible variations a Lagrangian plane. In the case of optimal control problems proposition~\ref{prop:l_def} gives an effective definition;
\item A $\cL$-derivative exists if the restriction of the Hessian at the considered critical point has a finite positive or negative inertia index;
\item The set of all $\cL$-derivatives over all subspaces of variations contains all the information about nullity, positive and negative inertia indices. By comparing relative positions of the corresponding Lagrangian planes one can track how those numbers change as we add variations;
\item $\cL$-derivatives do not change if we replace an infinite-dimensional space of variations by some dense subspace.
\end{enumerate}

Our goal is to specialize the definition of the Jacobi curve $\cL_t$ by precomputing the corresponding generalized limits. This way we are going to obtain simple constructions of $\cL_t$ in several cases. As a result we will recover some already known results and a completely new dynamical systems characterization of the Jacobi curve in the presence of a singularity that would be hard to guess without this general definition. Using this we can construct Jacobi curves of combinations of different extremals.

\section{Jacobi DE under the strengthened generalized Legendre condition}
\label{sec:jacobide}
Let us assume that the control $\tilde{u}(t)$ takes values on a one-dimensional edge of $U$. Our goal is to give a characterization of the Jacobi curve using an ODE, like in the classical theory. As we have discussed in the previous sections we will only use variations of the control parameter $u$ to construct the corresponding Jacobi curve.

We consider a sequence of functions $b^i(\tau)$ that we define as
$$
b^i(\tau) = \begin{cases}
b(\tau), & \text{ if } i=0, \\ 
\sigma\left( X^{(i)}(\tau),X^{(i-1)}(\tau) \right), & \text{ if } i \geq 1.
\end{cases}
$$
For the sake of simplicity we will often drop in the future the explicit dependence on time $\tau$ and simply write $b^i$ or $\sigma\left( X^{(i+1)},X^{(i)} \right)$, when there is no confusion.

The \textit{strengthened Legendre condition of order $m$} is a series of identities of the form
$$
b^m \leq \beta <0, \qquad b^j \equiv 0, \qquad j<m
$$
for some $m\in \Z_{\geq 0}$, where $\beta$ is just a constant. We say that an extremal curve $\tilde{q}(\tau)$ is a \textit{singular curve of order $m$}, if along it the strengthened Legendre condition of order $m$ is satisfied. If along a trajectory $b^i \equiv 0$ for all $i\in \Z_{\geq 0}$, we say that the trajectory has order infinity.

We define the \textit{Goh subspaces} as
$$
\Gamma^i(\tau) = \spn \{X^{(j)}(\tau):j\leq i\}.
$$
\begin{lemma}
\label{lemm:isotropic}
Assume that the strengthened Legendre condition of order $m$ is satisfied along an extremal trajectory $\tilde{q}(\tau)$. Then $\Gamma^{m-1}(\tau)$ is an isotropic subspace. Moreover $X^{m}(\tau) \in \Gamma^{m-2}(\tau)^\angle$.
\end{lemma}
\begin{proof}
The proof is a simple inductive argument. For $i = 1$ the statement is obvious since $\Gamma^1= \R X$. Assume that the statement is true for $i < m-1$. Then in particular we have 
$$
\sigma\left(X^{(i)}, X^{(j)}\right) = 0, \qquad \forall j < i.
$$
Differentiating this identity and using the induction assumption we find that
$$
\sigma\left(X^{(i+1)}, X^{(j)}\right) = 0, \qquad \forall j<i.
$$

The fact that $X^{m}(\tau) \in \Gamma^{m-2}(\tau)^\angle$ now follows from the differentiation of $$\sigma(X^{(m-1)},X^{(j)}) = 0, \qquad \forall j< m-2. $$
\end{proof} 

We now prove the following characterization of the Jacobi curve.
\begin{theorem}
\label{eqjacobi_de}
Let $\tilde q(\tau)$ be a regular or a singular extremal of order $m$. Then $\cL_t$ for $t>0$ is a linear span of $\Gamma^{m-1}(t)$ and the solutions of the following linear ODE
$$
\dot{\mu} = \frac{\sigma(X^{(m)},\mu)}{b^m}X^{(m)},
$$
with boundary conditions $\mu(0) \in T_{\lambda(0)}(T^*_{q_0}M) \cap (\Gamma^{m-1}(0))^\angle $. 

If the trajectory has infinite order, then we can define
$$
\Gamma(\tau) = \bigcup_{i=1}^\infty \Gamma^i(\tau)
$$
and $\cL_t = (T_{\lambda(0)}(T^*_{q_0}M))^{\Gamma(0+)}$.
\end{theorem}

\begin{proof}
The proof is based on the technique called the Goh transformations. The idea of that technique is that if an extremal is singular, then the differentials of maps and quadratic forms in the definition of the $\cL$-derivative remain continuous in a much weaker topology. So we can extend this map by continuity to a bigger space, in which the original space is dense. From the Theorem~\ref{thm:main} we know that this will not change the $\cL$-derivative. 

Let us assume first that the extremal trajectory is regular. If a vector
$$
\eta(t) = \eta + \int_0^t X(\tau) v(\tau) d\tau, \qquad \eta \in \cL_0
$$
is in $\cL_t$, then it satisfies
\begin{equation}
\label{eq:gen-for}
\int_0^t \sigma \left( \eta + \int_0^\tau X(\theta) v(\theta) d\theta, X(\tau) w(\tau) \right) + b(\tau)v(\tau)w(\tau)d\tau  = 0, \qquad \forall w\in L^2[0,t]
\end{equation}
From (\ref{eq:gen-for}) we get that $v(\tau)$ must satisfy 
$$
\sigma (\eta(\tau), X(\tau)) + b(\tau)v(\tau) =0 \iff v(\tau) = -b(\tau)^{-1}\sigma(\eta(\tau),X(\tau)), \qquad \text{a.e.} \tau \in [0,t]
$$
But on the other hand from the definition of $\eta(\tau)$ we have
$$
\dot{\eta}(\tau) = X(\tau) v(\tau) \qquad \Rightarrow \qquad \dot{\eta}(\tau) = - X(\tau) b(\tau)^{-1}\sigma(\eta(\tau),X(\tau))
$$
which gives us the classical Jacobi equation~\cite{as}.

We assume now that the extremal is singular of order $m$. It is clear that this derivation is not going to work anymore, since $b(\tau)  \equiv 0$, so we modify it in the following way. We denote 
$$
P^m v(t) = \int_0^t \int_0^{\tau_1}... \int_0^{\tau_{m-1}}v(\tau_m)d\tau_m ... d\tau_1, \qquad P^m w(t) = \int_0^t \int_0^{\tau_1}... \int_0^{\tau_{m-1}}w(\tau_m)d\tau_m ... d\tau_1
$$
the $m$-th primitives of $v$ and $m$. We integrate by parts $m$ times the first summand of (\ref{eq:gen-for})
$$
\sigma\left( \eta, \int_0^t X(\tau) w(\tau) d\tau \right) = \sigma \left( \eta, \sum_{i=0}^{m-1} (-1)^{i} X^{(i)}(t)(P^{i+1} w(t)) +(-1)^m \int_0^t X^{(m)}(\tau) (P^m w(\tau)) d\tau  \right)
$$
Now we integrate by parts the other summand of (\ref{eq:gen-for}). Exchanging the order of integration, using  Lemma~\ref{lemm:isotropic} and the assumption on the order of our extremal curve:
\begin{align*}
&\int_0^t \sigma \left( \int_0^\tau X(\theta) v(\theta) d\theta, X(\tau) w(\tau) \right) d\tau = \textrm{(integration by parts)} =\\
=&\int_0^t \sigma\left( X(\tau) (Pv(\tau)) - \int_0^\tau \dot{X}(\theta) (Pv(\theta)) d\theta, X(\tau) w(\tau) \right) d\tau = \textrm{(exchanging order)} =\\
=& \int_0^t \sigma\left(  - \dot{X}(\tau) (Pv(\tau)), \int_\tau^t X(\theta) w(\theta)  d\theta \right) d\tau = \textrm{(integration by parts)} = \\
= &\int_0^t \sigma\left(  - \dot{X}(\tau) (P v(\tau)), X(t) (P w(t)) - X(\tau) (Pw(\tau)) - \int_\tau^t \dot X(\theta) (Pw(\theta))  d\theta \right) d\tau = \textrm{(assumption)} = \\
= & \sigma\left(  - \int_0^t \dot{X}(\tau) (P v(\tau)) d\tau, X(t) (P w(t))\right)  + \int_0^t \sigma\left(  \int_0^\tau \dot{X}(\theta) (P v(\theta))  d\theta, \dot X(\tau) (P w(\tau))  \right) d\tau
\end{align*}
We continue to integrate by parts both summands and use Lemma~\ref{lemm:isotropic}, until the Legendre term $\sigma(X^{(m)}(t),X^{(m-1)}(t))$ will not appear explicitly. At the end we get
\begin{align*}
&\int_0^t \sigma \left( \int_0^\tau X(\theta) v(\theta) d\theta, X(\tau) w(\tau) \right) d\tau = \\
= & \sigma\left(  (-1)^m  \int_0^t X^{(m)}(\tau) (P^m v(\tau)) d\tau, \sum_{i=0}^{m-1} (-1)^{i}X^{(i)}(t) (P^{i+1} w(t))\right) +\\
+ & \int_0^t \sigma \left( X^{(m)}(\tau) (P^mv(\tau)), X^{(m-1)}(\tau) (P^mw(\tau) \right)d\tau +\\
+ &
\int_0^t \sigma\left(  \int_0^\tau X^{(m)}(\theta) (P^mv(\theta))  d\theta, X^{(m)}(\tau) (P^mw(\tau))  \right) d\tau
\end{align*}
Thus (\ref{eq:gen-for}) is transformed into
\begin{align}
&\sum_{i= 0}^{m-1}\sigma \left( \eta + (-1)^m \int_0^t X^{(m)}(\tau)(P^mv(\tau)) d \tau, (-1)^i X^{(i)}(t)(P^{i+1}w(t))  \right) +\nonumber\\
 &+\int_0^t \sigma \left( (-1)^m \eta + \int_0^\tau X^{(m)}(\theta) (P^m v(\theta)) d\theta, X^{(m)}(\tau) (P^m w(\tau) \right) d \tau + \label{eq:gen-for-k}\\
 &+ \int_0^t \sigma \left( X^{(m)}(\tau) (P^m v(\tau)), X^{(m-1)}(\tau) (P^mw(\tau)) \right) d\tau = 0 \nonumber
\end{align}

 We also integrate by parts the integral representation of $\eta(t)$, to get
$$
\eta(t) = \eta + \sum_{i=0}^{m-1} (-1)^{i}X^{i}(t) (P^{i+1} v(t)) + (-1)^m \int_0^t X^{(m)}(\tau)(P^mv(\tau)) d\tau
$$

We can see that the right hand side of this expression and quadratic form in  (\ref{eq:gen-for-k}) are continuous in the topology $\hat{H}^{-m}[0,t]$ given by the norm
$$
||v||_{-m} = \sqrt{\sum_{i=0}^{m-1} (P^i v(t))^2 + ||P^m v||^2_{L^2}}.
$$
So we extend by continuity on $\hat{H}^{-m}[0,t]$.
It is important to note that in the $\hat{H}^{-m}[0,t]$ the end-points $P^iv(t)$ represent separate variables. This implies immediately that $\Gamma^{m-1}(t) \subset \cL_t$. Indeed, we can see that the right-hand side of (\ref{eq:gen-for-k}) does not depend on $P^i v(t)$ at all. So if we take $\eta = 0$ and $P^m v(\tau) \equiv 0$, then (\ref{eq:gen-for-k}) will be satisfied automatically. But then $\eta(t) \in \Gamma^{m-1}(t)$ and every vector of $\Gamma^{m-1}(t)$ can be realized this way.

This means that $\cL_t$ actually consists of vectors
$$
\mu(t) = \eta + (-1)^m \int_0^t X^{(m)}(\tau)(P^m v(\tau)) d\tau
$$
and vectors from $\Gamma^{m-1}(t)$.

The derivative of $\mu(\tau)$ is given by
$$
\dot{\mu}(\tau) = (-1)^m X^{(m)}(\tau)(P^mv(\tau)).
$$
We only need to find $P^m v(\tau)$. We do this by solving (\ref{eq:gen-for-k}), which gives us a system of equations
\begin{align}
&\sigma \left( \mu(t),  X^{(i)}(t)\right) = 0, \qquad 0\leq i \leq m-1, \label{eq:boundary} \\
&(-1)^m\sigma \left( \mu(\tau), X^{(m)}(\tau)\right) +  b^m(\tau) P^mv(\tau) = 0, \qquad \text{a.e. } \tau \in [0,t].
\end{align}
Then from the last equation we recover
$$
P^mv(\tau) = (-1)^{(m+1)} (b^m(\tau))^{-1}\sigma \left( \mu(\tau), X^{(m)}(\tau)\right) 
$$
and so
$$
\dot{\mu}(\tau) = - X^{(m)}(\tau) (b^m(\tau))^{-1}\sigma \left( \mu(\tau), X^{(m)}(\tau) \right).
$$
From (\ref{eq:boundary}) we recover boundary conditions
\begin{equation}
\label{eq:boun-resol}
\mu(t) \in \Gamma^{m-1}(t)^\angle.
\end{equation}

We can prove that this identity is true not only for the chosen time $t$, but for any $\tau\in [0,t]$. Indeed, from the explicit form of the Jacobi DE we can see that $X^{(m-1)}(\tau)$ is a particular solution. But since all solutions lie in $\cL_\tau$ we have $\sigma(\mu(\tau), X^{(m-1)}(\tau)) = 0$ for any solution $\mu(\tau)$. Assume that the same is true for $X^{(i)}(\tau)$, $i\leq m-1$. Then for $X^{(i-1)}(\tau) $ we have
$$
\frac{d}{d\tau}\sigma\left( \mu(\tau),  X^{(i-1)}(\tau)\right) =- \frac{\sigma \left( X^{(m)}(\tau),\mu(\tau) \right)}{b^m(\tau)}\sigma \left( X^{(m)}(\tau), X^{(i-1)}(\tau) \right) + \sigma\left( \mu(\tau),X^{(i)}(\tau) \right) \equiv 0
$$
by Lemma~\ref{lemm:isotropic} and the induction assumption.

So we see that $\mu(\tau) \in \Gamma^{m-1}(\tau)^\angle$ is satisfied automatically if $\mu(0+) = (T_{\lambda(0)}(T^*_{q_0}M))^{\Gamma^{m-1}(0)}$. It means that by fixing the appropriate boundary conditions all $n$ independent solutions will lie in $\cL_t$.

Let us now look at what can happen if the singularity is of an infinite order. Since we have already established that $\Gamma(\tau)$ is an isotropic subspace, its dimension is limited. This can happen only if higher derivatives of $X$ become dependent from the lower derivatives. Let us assume that the first $l$ derivatives are generically independent and the $(l+1)$-th is not. Then $\Gamma^l(\tau)$ is a fixed subspace. Indeed, we can represent $\Gamma^l(\tau)$ as an element of $\wedge^l \R^{2n}$
$$
\Gamma^l(\tau) = X(\tau) \wedge \dot{X}(\tau) \wedge ... \wedge X^{(l)}(\tau).
$$
Then $\dot{\Gamma}^l(\tau) = \kappa(\tau) \Gamma^l(\tau)$ for some function $\kappa(\tau)$. Since $\wedge^l \R^{2n}$ is a linear space, the solution of this equation is simply
$$
\Gamma^l(\tau) = e^{\int_s^\tau \kappa(\theta) d\theta} \Gamma^l(s).
$$
So we see that $\Gamma(\tau) =\Gamma$ is constant except maybe a finite number of points, where the first $l$ derivatives of $X$ can become dependent. Therefore $\Gamma$ can be taken to be equal to $\Gamma(0+)$. 

By assumption on the infinite order we have that the elements of $\cL_t$ must satisfy
$$
\int_0^t \sigma \left( \eta + \int_0^\tau X^{(l)}(\theta) P^l v(\theta), X^{(l)}(\tau) w(\tau) \right)d\tau = 0.
$$   
We look for a solution with $P^l v(\tau) = 0$, then the equation above is transformed to
$$
\sigma \left( \eta ,\int_0^t  X^{(l)}(\tau) w(\tau) d\tau \right)= 0.
$$   
But we have seen that $X^{(l)}(\tau)$ never leaves $\Gamma$. Therefore all $\eta \in (T_{\lambda(0)}(T_{q_0}^*M))\cap\Gamma^\angle$ satisfy the equation above as well as the boundary conditions \eqref{eq:boundary}. And as a consequence those vectors together with vectors from $\Gamma(\tau)$ give $n$ independent solutions whenever $\dim(\Gamma(\tau))$ is maximal. At the isolated points where the dimension of this space drops we simply use the left-continuity property from  Lemma~\ref{lem:continuity}.
\end{proof}

\section{Bang-bang extremals}
\label{sec:bang}
If the extremal curve is bang-bang, then the extremal control takes values on the vertices of $U$. In this case we use time variations to construct the Jacobi curve. 

Let 
$$
f(q,\tilde{u}(t)) =: f_i(q), \qquad L(q,\tilde{u}(t)) =: L_i(q)\qquad t\in [t_i,t_{i+1}].
$$
be the controled system and the minimized functional on an interval of constancy $[t_i,t_{i+1}]$ of $u(t)$. Since the system under consideration was autonomous we have that $X(\tau)$ is also a piece-wise constant function. We define $X_i = X(\tau)$, $\tau\in (t_i,t_{i+1}]$. These $X_i$ have a particularly nice form, when we consider an optimal time problem. In this case 
$$
(P^{\tau}_0)^{-1}_* f(q,\tilde{u}\tau) = 
e^{-t_1 f_1}_*(e^{(t_1 - t_2)f_2}_*(...(e^{(t_{i+1} - t_{i})f_{i-1}}_*f_i)...)), \qquad \tau \in (t_i,t_{i+1}]
$$
Let $h_i(\lambda) = \langle \lambda, (P^\tau)^{-1}_* f(q,\tilde{u}(\tau))\rangle$, $t\in(t_i,t_{i+1}]$. Then $X(t) = X_i =  \vec{h}_i$ for $t\in(t_i,t_{i+1}]$.

We can now apply Proposition~\ref{prop:sungle_u} to find an approximation of the Jacobi curve. We take $V_j$ to be the space of variations constant on the intervals $[t_i,t_{i+1}]$ and which are zero for $t\geq t_j$. We have $\cL_t(\{0\}) = T_{\lambda(0)}(T^*_{q_0} M)$ and it is possible now to apply inductively Proposition~\ref{prop:sungle_u}.

Since the new system and the functional are linear in control, we obtain $b(\tau) \equiv 0$. Then 
$$
\sigma\left( \int_{t_i}^\tau X(\theta)d\theta, X(\tau) \right) + b(\tau) = (\tau - t_i)\sigma(X_i,X_i) = 0, \qquad \forall \tau \in [t_i,t_{i+1})
$$
and so $K = 0$ on each step. Similarly we have
$$
\eta(t_{i+1}) = \frac{1}{t_{i+1}-t_i} \int_{t_i}^{t_{i+1}} X(\tau) d\tau = X_i.
$$
This way we obtain a sequence of Lagrangian subspace $\cL_t(V_j)$, defined inductively as
$$
\cL_t(V_0) = \cL_t(\{0\}) = T_{\lambda(0)}(T^*_{q_0} M), \qquad \cL_t(V_{i+1}) = \cL_t(V_i)^{X_i}.
$$
If we take a finer splitting of the interval the corresponding approximation to the Jacobi curve is the same as above, because $(\cL_t(V_i)^{X_i})^{X_i} = \cL_t(V_i)^{X_i}$. 

The final algorithm of constructing the Jacobi curve goes as follows. One defines $\cL_0 = T_{\lambda(0)}(T^*_{q_0}M)$. The Jacobi curves $\cL_\tau$ is constant for $\tau\in(t_i,t_{i+1}]$ and after a switching it jumps to $\cL(t_{i} +) = \cL(t_i)^{X_{i}}$. This is the same algorithm that was obtained in~\cite{agrachev_bang}.

We note that Lemma~\ref{lemm:add} allows to generalize the previous discussion to a wider range of situations. For example, using the results of the previous two sections we can treat the Fuller phenomena, at least when bang-bang arcs are followed by a non-degenerate singular arc. We know from the bang-bang algorithm that in order to construct the Jacobi curve we have to put $X(\tau)$ at each switching time inside of the $\cL$-derivative. Left-continuity ensures that this procedure will give a convergent sequence if the index of the Hessian is finite. After the limit $\Lambda$ of the corresponding planes has been found, we can define the Jacobi curve using the flow of the Jacobi DE with the right boundary conditions. From Lemma~\ref{lemm:add} it follows that we have to replace in Theorem~\ref{eqjacobi_de} $T_{\lambda(0)}(T^*_{q_0}M)$ with $\Lambda$.

\section{Normal form for the Jacobi DE for the simplest singularity}
\label{sec:normal_form}
In this section we consider the simplest singularity when $b(\tau) = 0$ for some moment of time $\tau$. Due to analyticity assumption such a moment of time must be isolated. We would like to construct the Jacobi curve after we pass the singularity. In the following sections we are going to prove the following result.
\begin{theorem}
\label{thm:jump}
Let $\tilde{q}(\tau)$ be an extremal curve for which $b(\tau) = 0$ and $b(\tau+\varepsilon)=b_m\varepsilon^m + ...$ is negative for all $\varepsilon>0$ sufficiently small. Assume that at $\tau$ the following conditions are satisfied 
\begin{enumerate}
\item $\sigma(X(\tau),\dot{X}(\tau))\neq 0$ and $4\sigma(X(\tau),\dot{X}(\tau))+b_2 \neq 0 $ if $m=2$;
\item $\dim\spn\{\ddot X(s), \dot X(s), X(s)\} = const$ for $s\in[\tau,\tau +\varepsilon]$.
\end{enumerate}
Then if the right limit of the Jacobi curve $\cL(s)$ at $s=\tau$ exists, it is equal to
\begin{equation}
\label{eqboundary_thm}
\cL(\tau+) = \cL_\tau^{X(\tau)}.
\end{equation}
\end{theorem}

Before we proceed we would like to make some remarks about the statement of the theorem. 
\begin{enumerate}
\item The assumptions we make allow us to have a simplest possible singularity that is in some sense is generic. It is possible to replace those conditions with different ones and to study even more singular cases;
\item The Jacobi curve may not be well defined, due to an infinite inertia index of the Hessian. We will give sufficient conditions for existence and non-existence using oscillation theorems for Hamiltonian systems in Section~\ref{sec:kneser};
\item It is not true that after the singularity the Jacobi curve is determined only by its jump. Indeed, it must satisfy the Jacobi equation, but at the same time the right hand-side of the Jacobi equation is not even continuous. So we do not have neither existence nor uniqueness of solutions and we need more information to isolate the right solution. In Section~\ref{sec:jumpg2} for $m=1,2$ we will prove that Jacobi curve can be uniquely characterized by a one-jet. 
\end{enumerate}

Due to analyticity singular points can not cluster. That is why without any loss of generality from now on we assume that $b(0) = 0$ and that $b(\tau) <0$ for $\tau$ sufficiently small. To give a characterization of the Jacobi curve after a singularity of the considered type, we use once again the theory of $\cL$-derivatives. 

Even in this case we still have a Jacobi equation of the form
$$
\dot{\eta} = 
\frac{\sigma(X,\eta)}{b}X,
$$
and it governs the behaviour of the Jacobi curve away from singularity. In order to understand how to proceed at the singularity, we must recall that by definition  $\cL$-derivatives are constructed by adding more and more variations and in the limit we get pointwise convergence to the Jacobi curve. Assume that we use only variations whose support does not intersect $[0,\varepsilon]$. Then using the same argument as in the previous section we obtain a slightly different Jacobi equation of the form
$$
\dot{\eta} = \begin{cases}
\frac{\sigma(X,\eta)}{b}X, & \text{ if } \tau\notin[0,\varepsilon], \\ 
0, & \text{ if } \tau\in[0,\varepsilon].
\end{cases}
$$
It is equivalent to the following construction. We will denote by $\Lambda^\varepsilon(\tau)$ a solution of the induced Jacobi equation in the Lagrangian Grassmanian. We use the Jacobi flow to determine the Jacobi curve until time $0$. We assume that the left limit exist and is equal to the corresponding $\cL$-derivative $\Lambda^\varepsilon(0) = \cL_0$. Then the flow does nothing for a while, meaning that $\Lambda^\varepsilon(\tau) = \cL_0$ for $\tau \in[0,\varepsilon]$ and then we continue with the flow after the moment of time $\varepsilon$, where the dynamics is non-singular until the next zero of $b(\tau)$. This means that the Jacobi curve is going to be a point-wise limit of solutions of the Jacobi equation on the Lagrangian Grassmanian that satisfy $\Lambda^\varepsilon(\varepsilon) = \cL_0$. Since outside of the singularity we have uniqueness and existence, for each $\varepsilon$ we obtain a unique curve in $L(n)$. The pointwise limit of these curves is the Jacobi curve we seek.

To realize this strategy we first simplify the Jacobi DE and reduce the dimension of the considered problem by separating singular and regular dynamics of the Jacobi equation. Let $J$ be the complex structure associated to the symplectic form $\sigma$. Then we can rewrite the Jacobi DE as
$$
\dot{\eta} = \frac{X X^T J }{b}\eta.
$$
We make a time-dependent change of variables $\mu(\tau) = M^{-1}(\tau) \eta(\tau)$. We get
$$
\dot{\mu} = -M^{-1} \dot{M} \mu + \frac{M^{-1} X X^T J M }{b}\mu.
$$
First we look carefully at the second term. We assume that the matrix $M$ is symplectic. Then $M^{-1} = -JM^T J$ and we obtain
$$
M^{-1}X X^T J M = (M^{-1} X )(M^T J^T X)^T = (M^{-1} X )(J^T J M^T J^T X)^T = (M^{-1} X)(M^{-1} X)^T J.
$$
We also make a choice for the first column of $M$ by assuming 
$$
M^{-1}(\tau)X(\tau) = \begin{pmatrix}
1 \\
0 \\
\vdots\\
0
\end{pmatrix}.
$$
All this implies that we get an equation of the form
\begin{equation}
\label{eq_intermid}
\dot{\mu}= -M^{-1}\dot{M} \mu + \frac{1}{b}
\sbox0{$\begin{matrix}1&0&\cdots &0\\0&0&\cdots &0\\\vdots &\vdots &\ddots &\vdots \\0&0&\cdots &0\end{matrix}$}
\left(
\begin{array}{c|c}
\makebox[\wd0]{\large $0$}&\usebox{0}\\
\hline
  \vphantom{\usebox{0}}\makebox[\wd0]{\large $0$}&\makebox[\wd0]{\large $0$}
\end{array}
\right)
\mu.
\end{equation}
Now we work with the first term $M^{-1}\dot{M}$. Since we choose $M$ to be symplectic, its columns form a Darboux basis. Denote the first $n$ columns by $e_i$ and the last by $f_i$. Then we can write
$$
-M^{-1}\dot{M} = \begin{pmatrix}
\sigma(f,\dot{e}) & \sigma(f,\dot{f}) \\
\sigma(e,\dot{e}) & \sigma(e,\dot{f})
\end{pmatrix} = \begin{pmatrix}
\sigma(f,\dot{e}) & \sigma(f,\dot{f}) \\
\sigma(e,\dot{e}) & -\sigma(f,\dot{e})
\end{pmatrix},
$$
where $\sigma(x,y)$ means a matrix whose elements are $\sigma(x_i,y_j)$, for two $n$-tuples of vectors $x=(x_1,...,x_n)$ and $y=(y_1,...,y_n)$. The last equality follows from the fact that the basis is Darboux, i.e.
$$
\sigma(e,f) = \id_n \Rightarrow \sigma(\dot{e},f) + \sigma(e,\dot f) = 0.
$$

Let us first assume that $n = 1$. Since by assumption of Theorem~\ref{thm:jump} $\sigma(X(0),\dot{X}(0))\neq 0$, we can choose
\begin{equation}
\label{eq:choice1}
e(\tau) = e_1(\tau) = X(\tau), \qquad f(\tau)= f_1(\tau) = \frac{\dot{X}(\tau)}{\sigma(X(\tau),\dot{X}(\tau))},
\end{equation}
The Jacobi DE reduces then to the following normal form
\begin{equation}
\label{eqsing1}
\frac{d}{d\tau}\begin{pmatrix}
\mu_1 \\
\mu_{n+1}
\end{pmatrix} = 
\begin{pmatrix}
0 & \dfrac{\sigma(\dot X,\ddot X)}{\sigma(X,\dot X)^2} + \dfrac{1}{b} \\
\sigma(X,\dot X) & 0
\end{pmatrix}
\begin{pmatrix}
\mu_1 \\
\mu_{n+1}
\end{pmatrix}
\end{equation}

Let us now assume, that $n\geq 2$. We want to separate the singular dynamics from the regular dynamics. If we look at the singular part in \eqref{eq_intermid}, then we see that the only non-zero element is in the first row and $(n+1)$ column. So in the new coordinates we want at least some of the expressions for $\dot{\mu}^i$ for $i\neq 1,n+1$ to be independent of $\mu^1, \mu^{n+1}$. Moreover the assumption, that $M(\tau)$ is symplectic, is going to imply in addition that we will have two invariant symplectic subspaces: a subspace containing $\mu_1,\mu_{n+1}$ coordinates, where the singular dynamics happens and its complement where the dynamics is smooth.

So we look for $e_i$, $f_i$ such that
\begin{equation}
\label{eq:cond1}
\sigma \left(e_i,e_1 \right)  = 0, 
\qquad 
\sigma \left(f_i,e_1 \right)  = 0,
\qquad 
\sigma \left(e_i,f_1 \right)  = 0,
\qquad 
\sigma \left(f_i,f_1 \right)  = 0.
\end{equation}

From the assumption 2) in Theorem~\ref{thm:jump} that we have made it follows that the dimensions of the space $\spn\{\ddot X(\tau), \dot X(\tau), X(\tau)\}$ must be equal either to two or three. In the first case $\ddot{X}(\tau)$ is simply in the span of $X(\tau)$, $\dot{X}(\tau)$ for small $\tau\geq 0$. So we make the same choice (\ref{eq:choice1}) for $e_1, f_1$ and the rest of the columns we take to be a smooth Darboux basis for the symplectic space $\spn\{X(\tau),\dot X(\tau)\}^\angle$. Then the conditions (\ref{eq:cond1}) are indeed satisfied and we obtain exactly the equation (\ref{eqsing1}) for the singular part.

In the second case we can not guarantee that conditions (\ref{eq:cond1}) are satisfied for $i\geq 2$, since now $\ddot{X}(\tau)$ has to be accounted for. To isolate the singular dynamics we choose a Darboux basis $e_1,e_2,f_1,f_2$ as follows. The vectors $e_1,f_1$ are as before, $e_2$ is defined as
$$
e_2  = \ddot{X} - \frac{\sigma(\ddot{X},X)}{\sigma(\dot{X},X)}\dot{X} + \frac{\sigma(\ddot{X},\dot X)}{\sigma(\dot{X},X)}X.
$$
and $f_2$ is chosen to be any vector such that we get a Darboux basis. In this case it just means that
$$
\sigma(X,f_2) = \sigma(\dot X,f_2) = \sigma(\ddot X,f_2) -1 = 0.
$$
The rest of the columns of $M_\tau$ are chosen to be a smooth Darboux basis of the symplectic space $(\spn\{e_1,e_2,f_1,f_2\})^\angle$. Again, the derivatives of $e_1$, $f_1$ are contained in $\spn\{e_1,e_2,f_1,f_2\}$, so the dynamics splits. The singular dynamics takes place in the plane with $(\mu_1,\mu_2,\mu_{n+1},\mu_{n+2})$ coordinates. Thus we get an invariant subsystem of dimension two which describes all the singular dynamics
$$
\frac{d}{dt}\begin{pmatrix}
\mu_1 \\
\mu_2 \\
\mu_{n+1} \\
\mu_{n+2}
\end{pmatrix} = 
\begin{pmatrix}
\sigma(f_1,\dot e_1) & \sigma(f_1,\dot e_2) & \sigma(f_1,\dot{f}_1) + \frac{1}{b} & \sigma(f_1,\dot{f}_2)\\
\sigma(f_2,\dot e_1) & \sigma(f_2,\dot e_2) & \sigma(f_2,\dot{f}_1) & \sigma(f_2,\dot{f}_2)\\
\sigma(e_1,\dot e_1) & \sigma(e_1,\dot e_2) & \sigma(e_1,\dot{f}_1) & \sigma(e_1,\dot{f}_2) \\
\sigma(e_2,\dot e_1) & \sigma(e_2,\dot e_2) & \sigma(e_2,\dot{f}_1) & \sigma(e_2,\dot{f}_2)
\end{pmatrix}
\begin{pmatrix}
\mu_1 \\
\mu_2 \\
\mu_{n+1} \\
\mu_{n+2}
\end{pmatrix}.
$$
and so we have proven our first result about the jump of the Jacobi curve $\cL_\tau$
\begin{proposition}
If $b(\tau) = 0$, then $\dim (\cL_\tau\cap \cL_{\tau+}) \geq n-2$.
\end{proposition}

Let us simplify this equation even more. Since $M^{-1}\dot M$ is a matrix from the symplectic Lie algebra, not all the entries above are independent. More precisely, the first and the last diagonal 2x2 minors are minus transpose of each other and the off diagonal 2x2 minors are symmetric. We can find explicitly
$$
\sigma(f_1,\dot{e}_1) = \sigma(f_1,\dot{e}_2) = \sigma(f_2,\dot{e}_1) = \sigma(e_1,\dot{e}_2)  = 0;
$$
$$
\sigma(e_1,\dot{f}_1) = \sigma(e_1,\dot{f}_2) = \sigma(e_2,\dot{f}_1) = \sigma(e_2,\dot{e}_1) = 0;
$$
$$
\sigma(f_2,\dot{f}_1) = \sigma(f_1,\dot{f}_2) = -\frac{1}{\sigma(X,\dot X)};
$$
$$
\sigma(f_2,\dot e_2) = - \sigma(e_2,\dot f_2) = \sigma(f_2,\dddot{X}) + \frac{\sigma(\ddot X, X)}{\sigma(\dot X, X)};
$$
$$
\sigma(e_1,\dot{e}_1) = \sigma(X,\dot X);
$$
$$
\sigma(e_2,\dot{e}_2) = \sigma(\ddot X,\dddot X) - \frac{\sigma(\dddot X,X)\sigma(\ddot X,\dot X)+\sigma(\ddot X,X)\sigma(\dddot X,\dot{X})}{\sigma(\dot X,X)};
$$
$$
\sigma(f_1,\dot{f}_1) = \dfrac{\sigma(\dot X,\ddot X)}{\sigma(X,\dot X)^2}.
$$
So we get an equation of the form
\begin{equation}
\label{eq:jacobi_notfull}
\frac{d}{dt}\begin{pmatrix}
\mu_1 \\
\mu_2 \\
\mu_{n+1} \\
\mu_{n+2}
\end{pmatrix} = 
\begin{pmatrix}
0 & 0 & \sigma(f_1,\dot{f}_1) + \frac{1}{b} & \sigma(f_2,\dot{f}_1)\\
0 & \sigma(f_2,\dot e_2) & \sigma(f_2,\dot{f}_1) & \sigma(f_2,\dot{f}_2)\\
\sigma(e_1,\dot e_1) & 0 & 0 & 0 \\
0 & \sigma(e_2,\dot e_2) & 0 & -\sigma(f_2,\dot{e}_2)
\end{pmatrix}
\begin{pmatrix}
\mu_1 \\
\mu_2 \\
\mu_{n+1} \\
\mu_{n+2}
\end{pmatrix}.
\end{equation}
Note that if $\sigma(f_2,\dot{f}_1)\equiv 0$ for $\tau$ small enough, we obtain the $n=1$ normal form of the Jacobi DE as a subsystem. So without any loss of generality from now on we can assume that $n\geq 2$.

We can simplify the last equation even more by taking
$$
Q(\tau) = \begin{pmatrix}
1 & 0 \\
0 & \exp\left( \int_0^\tau \sigma(\dot{f}_2(s),e_2(s)) ds \right)
\end{pmatrix}.
$$
Note that $Q(0) = \id_n$. We introduce new variables
$$
\begin{pmatrix}
p_1\\
p_2\\
q_1\\
q_2
\end{pmatrix} = 
\begin{pmatrix}
Q^{-1} & 0 \\
0 & Q
\end{pmatrix}
\begin{pmatrix}
\mu_1 \\
\mu_2 \\
\mu_{n+1} \\
\mu_{n+2}
\end{pmatrix}.
$$
If we write $p=(p_1,p_2)$, $q=(q_1,q_2)$, we obtain a normal form
\begin{equation}
\label{eqnormal_jacobi}
\frac{d}{dt}\begin{pmatrix}
p\\
q
\end{pmatrix} = 
\begin{pmatrix}
0 & \frac{B(\tau)}{\tau^m} \\
C(\tau) & 0 
\end{pmatrix}
\begin{pmatrix}
p\\
q
\end{pmatrix}
\end{equation}
where $m$ is the power of the first non zero coefficient of the Taylor expansion of $b(\tau)$,
$$
C(\tau) = Q\begin{pmatrix}
\sigma(e_1,\dot e_1) & 0\\
0 & \sigma(e_2,\dot e_2)
\end{pmatrix}Q
=
\begin{pmatrix}
c_{11}(\tau) & 0\\
0 & c_{22}(\tau) 
\end{pmatrix}
$$
and
$$
\frac{B(\tau)}{\tau^m} =  \begin{pmatrix}
1/b & 0 \\
0 & 0 
\end{pmatrix} 
+
 Q^{-1}\begin{pmatrix}
\sigma(f_1,\dot{f}_1)  & \sigma(f_2,\dot{f}_1)\\
\sigma(f_2,\dot{f}_1) & \sigma(f_2,\dot{f}_2)
\end{pmatrix}Q^{-1} = \begin{pmatrix}
\frac{1}{b(\tau)} + b_{11}(\tau) & b_{12}(\tau)\\
b_{12}(\tau) & b_{22}(\tau)
\end{pmatrix}
$$
Note that (\ref{eqnormal_jacobi}) is still a Hamiltonian system and a normal form for the singular part of the Jacobi DE. Moreover  we can choose our frame so that $B(\tau)$ is positive for small $\tau>0$. Indeed, this is going to be true if the trace and the diagonal entries are non-negative. Since $b(\tau) = b_m\tau^m + O(\tau^{m+1}) <0$ for $\tau>0$ small and the frame $\{e_i,f_i\}$ was chosen to be analytic, we obviously have that the trace and the first diagonal element are negative for $\tau>0$ sufficiently small. We claim that $f_2$ can be chosen in such a way that also the second diagonal term is negative as well for small $\tau>0$. Indeed, the only freedom that we have is to replace $f_2(\tau)$ with $f_2(\tau)+a(\tau)e_2(\tau)$ for some analytic function $a(\tau)$. Then we have
$$
\sigma(f_2+ae_2,\dot{f}_2 + \dot{a}e_2 + a\dot{e}_2) = \sigma(f_2,\dot{f}_2) -\dot{a} + a^2\sigma(e_2,\dot{e}_2),
$$
where we have used that $\sigma(e_2,f_2) = 1$ and $\sigma(\dot e_2,f_2) + \sigma(e_2, \dot{f}_2) = 0$. Then we can simply choose $a(t) = \sin kt$, with $k$ sufficiently large and the explicit form of $Q$ implies that we can assume without any loss of generality that $B(\tau)$ is negative for $\tau$ small.

Before we start proving Theorem~\ref{thm:jump}, it is very helpful to understand the idea of the proof using some simple heuristics in the case $n=1$. Later we will make all the steps rigorous. In the next section we will see under which conditions $\cL$-derivatives exist and how to characterize the Jacobi curve as a solution to the singular Jacobi equation with certain boundary conditions when $n=1$.  

\section{An heuristic argument for $n=1$}
\label{sec:easy_case}
Assume that $n=1$ and let $b_m \tau^m + ...$ be the right series of $b(\tau)$ at $\tau = 0$, with $b_m < 0$. We would like to determine whether or not  $\cL_{0+}$ exists at all. To do this we rewrite (\ref{eqsing1}) as a second order ODE of the form
\begin{equation}
\label{eqsing1_second_ord}
\ddot{\mu}_1 + a_1(\tau)\dot{\mu}_1+a_0(\tau)\mu_1 = 0.
\end{equation}
We say that this equation is oscillating on a given interval, if any solution $\mu_1(\tau)$ has an infinite number of zeroes on that interval. Equivalently the classical Sturm theory of second order ODEs implies that this equation is oscillating whenever any solution of (\ref{eqsing1}) makes an infinite number of turns around the origin in the $(\mu_1,\mu_{n+1})$ plane. Recall that for $n=1$ the Lagrange Grassmanian is nothing but a projective line $\P^1$ and that the Jacobi curve is just the line $\R(\mu_1,\mu_{n+1})$. Therefore Jacobi curves of oscillating equations  have an infinite Maslov index.

For second order ODEs there exist various oscillation and non-oscillation criteria, but among them there is a particularly simple one called the Kneser criteria~\cite{ode2}. It states that a second order ODE of the form 
\begin{equation}
\label{eqsturm}
\ddot x + a(s) x = 0
\end{equation}
is oscillating on $[p,\infty)$, for any $p>0$ if
$$
\lim_{s \to \infty} s^2 a(s) > \frac{1}{4}
$$
and it is non-oscillating if
$$
\lim_{s \to \infty} s^2 a(s) < \frac{1}{4}.
$$
If the limit is exactly $1/4$, Kneser criteria gives us no information and we have to use a different criteria. In order to put the equation (\ref{eqsing1_second_ord}) into form (\ref{eqsturm}), we simply make a change of the time variable $s = 1/\tau$ and after a change of the dependent variable
$$
x(s) =\mu(s) \exp \left( \int_{p}^s \frac{2\theta - a_1(\theta)}{2\theta^2} d\theta\right).
$$
Then we obtain exactly equation (\ref{eqsturm}). 

After applying the Kneser criteria, we find that for any sufficiently small interval $[0,\varepsilon]$
\begin{enumerate}
\item equation (\ref{eqsing1_second_ord}) is oscillating if $m=2$ and $4\sigma(X(0),\dot{X}(0))+b_2 > 0$ or if $m>2$ and $\sigma(X(0),\dot{X}(0)) > 0$;
\item equation (\ref{eqsing1_second_ord}) is non-oscillating if $1\leq m<2$ or $m=2$ and $4\sigma(X(0),\dot{X}(0))+b_2 < 0$, or if $m>2$ and $\sigma(X(0),\dot{X}(0)) <0$. 
\end{enumerate} 
Thus we conclude that in the first case the Jacobi curve has no right limit and therefore it also has an infinite Maslov index. 

Since we just want to give an idea of how the proof works, we assume that the Jacobi DE \eqref{eqsing1} is of the simplest form
\begin{equation}
\frac{d}{d\tau}\begin{pmatrix}
\mu_1 \\
\mu_{n+1}
\end{pmatrix} = 
\begin{pmatrix}
0 & \dfrac{1}{\tau^2} \\
C & 0
\end{pmatrix}
\begin{pmatrix}
\mu_1 \\
\mu_{n+1}
\end{pmatrix},
\end{equation}
where $C$ is constant. We then make a time-dependent change of variables
$$
\begin{pmatrix}
p\\
q
\end{pmatrix}
=
\begin{pmatrix}
\tau^{1/2} & 0 \\
0 & \tau^{-1/2}
\end{pmatrix}
\begin{pmatrix}
\mu_1\\
\mu_{n+1}
\end{pmatrix}
$$
and obtain
\begin{equation}
\label{eq_blowup_sing1}
\frac{d}{d\tau}\begin{pmatrix}
p\\
q
\end{pmatrix} = \frac{1}{\tau}\begin{pmatrix}
\frac{1}{2}	& 1 \\
C & -\frac{1}{2}
\end{pmatrix}
\begin{pmatrix}
p\\
q
\end{pmatrix} = -\tau^{-1}JH\begin{pmatrix}
p\\
q
\end{pmatrix}
\end{equation}
Note that this change of variables can not change whether or not the Maslov index is finite. It is clear that $\tau^{-1}$ multiplier just scales the speed along solutions, but does not change the trajectories. We could have actually got rid off it using a change of time variable. This means that if we drop $\tau^{-1}$ in the equation above, the overall phase-portrait does not change. It will be completely determined by the structure of the matrix $-JH$. 

A description of various phase portraits on the Lagrangian Grassmanian was given in~\cite{riccati}. We will use the results from that article to work out the general case. For example, for $n=1$ we can only have fixed points or periodic trajectories, which depend on the eigenvalues and eigenvectors of the matrix $-JH$. Its eigenvalues are
$$
\lambda_{1,2} = \pm\frac{\sqrt{1+4C}}{2}.
$$
If $1+4C < 0$, then we have only a single closed trajectory and no equilibrium points. Thus the trajectory rotates on $\P^1$, and because of the $\tau^{-1}$ multiplier in \eqref{eq_blowup_sing1} the curve rotates faster and faster as we get closer to $\tau = 0$ and therefore we get an infinite Maslov index. If $1+4C >0$, then we have two equilibrium points: a stable and a non-stable one, that are given by two lines spanned by the eigenvectors of $-JH$. Thus all solutions except the equilibrium ones tend to the unstable equilibrium as $\tau \to 0$ and to the stable one as $\tau \to \infty$. In this case the Maslov index is finite. Note that in our example $C = \sigma(X(0),\dot{X}(0))/b_2$ and thus we recover the classical Kneser criteria. From here we can also see very well why the case $1+4C =0$ is excluded. It is not stable under small perturbations and corresponds to a resonant situation when the two equilibrium points merge.  

Having a small dimensional situation allows us to actually draw the extend-phase portrait. We introduce new variables $U$ and $V$ defined as $\mu_{n+1}= U\mu_1$, $p=Vq$. It is clear from the definitions that $V = \tau U^{-1}$. Since we work in a coordinate chart of the Grassmanian we can assume that $U,V\neq 0$. We differentiate expressions in the definition to obtain a couple of related Riccati equations
$$
\dot U = C - \frac{U^2}{\tau^2},
$$
$$
\dot V = \frac{1+V-CV^2}{\tau}.
$$

\begin{figure}[ht!]  
\vspace{-4ex} \centering \subfigure[]{
\includegraphics[width=0.44\linewidth]{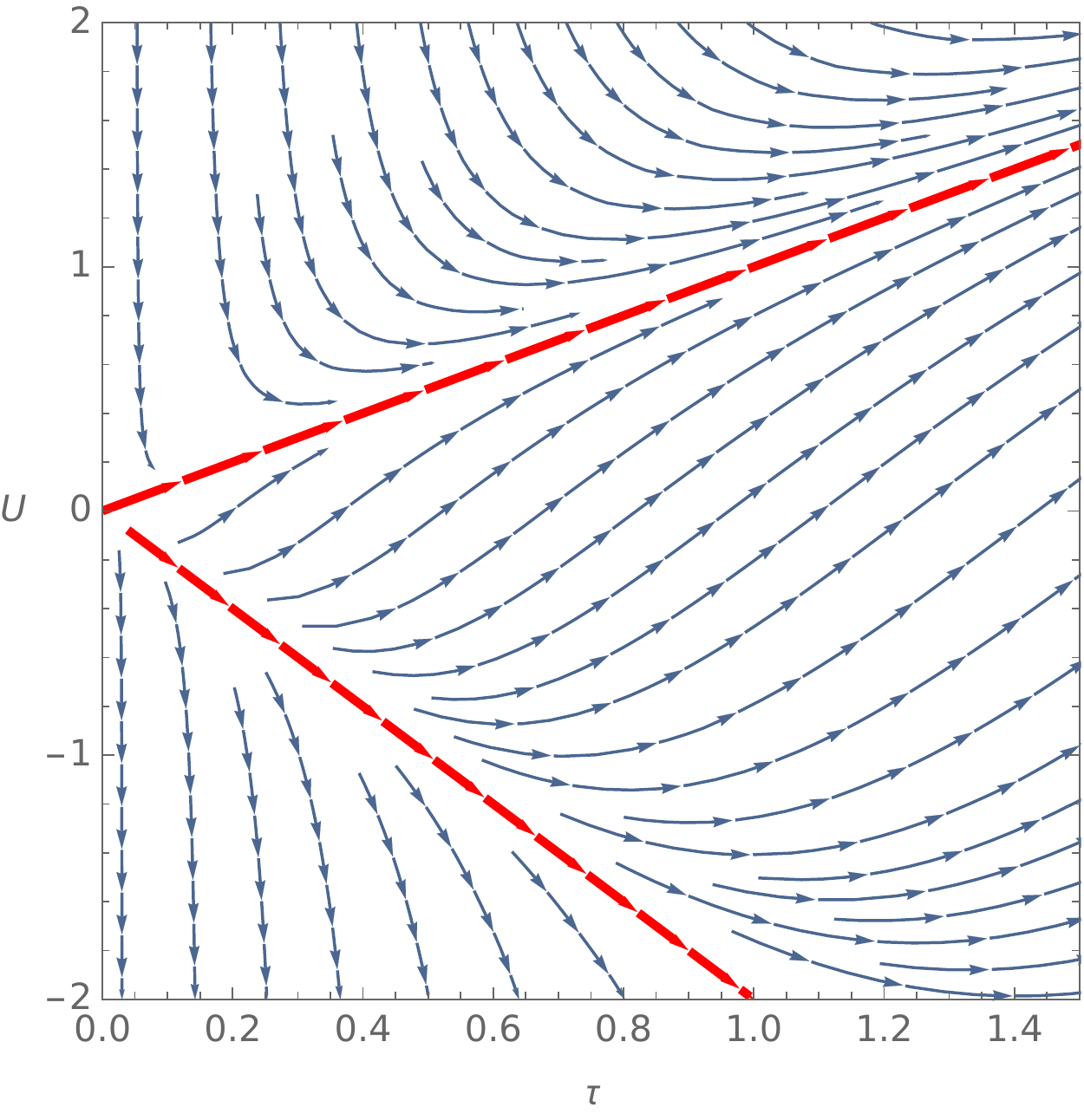} \label{fig:nososc_before} }  
\hspace{4ex}
\subfigure[]{
\includegraphics[width=0.44\linewidth]{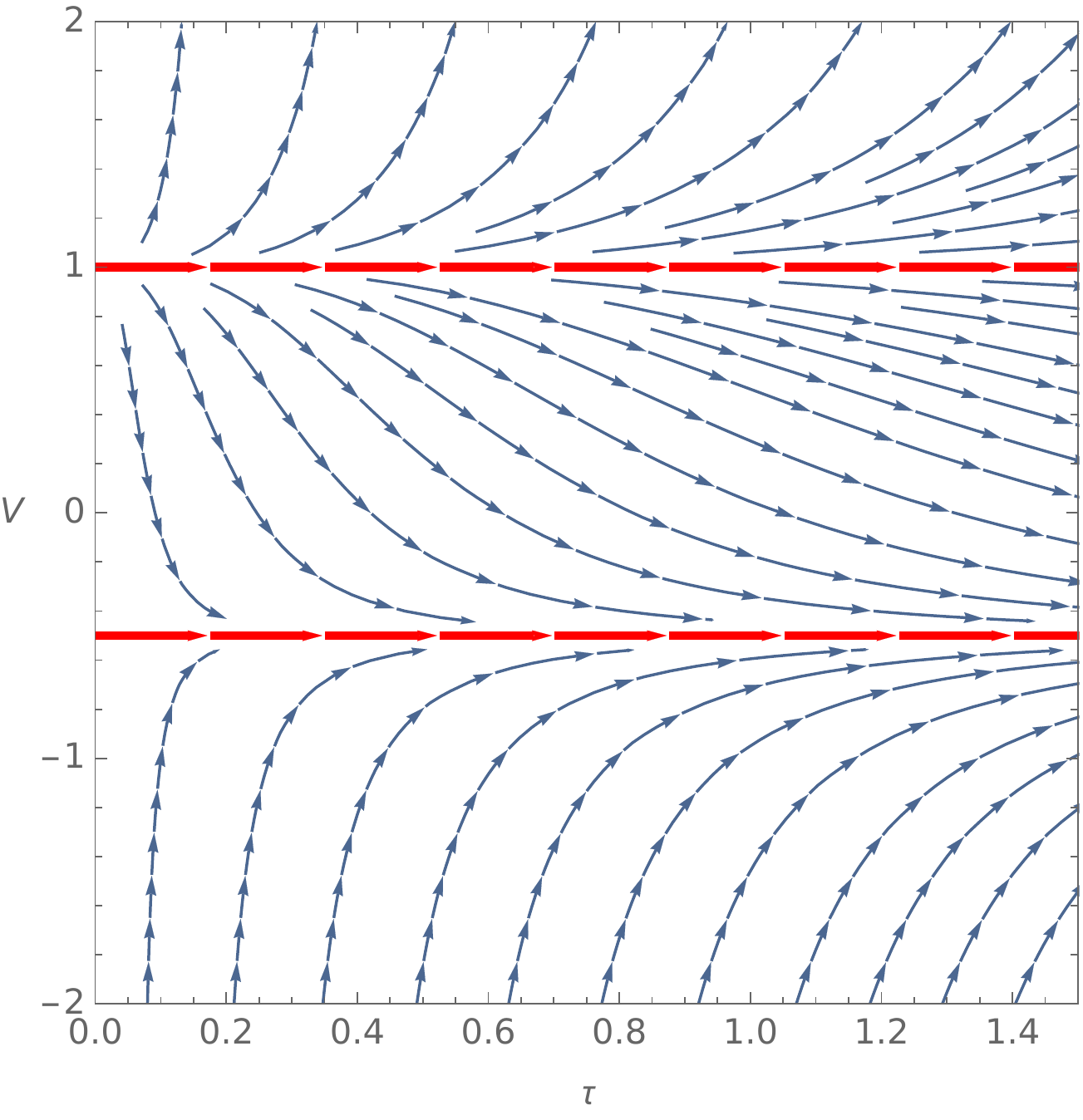} \label{fig:nososc_after} }  
\caption{Phase portrait of a non-oscillating system: \subref{fig:nososc_before} before blow-up; \subref{fig:nososc_after} after blow-up.} \label{fig:nonosc}
\end{figure}

In the picture~\ref{fig:nonosc} the non-oscillating extended phase portrait before and after the change of variables is depicted. We can see clearly, that the extended phase portrait is separated by two separatrix into three regions. After the blow up the two separatrix solutions have different initial values correspond to the equilibrium solutions. The stable solution after a blow-up can be described using an initial value problem, where as the unstable one can not. There is an infinite number of solutions that start from the unstable equilibrium.

We claim that the stable separatrix is the Jacobi curve for $\tau > 0$. To see this we do as discussed in the previous subsection. Assume that $\cL_0$ is given by $U_0\neq 0$. Then the Jacobi curve is the limit of solutions of the Riccati equations with $U(\varepsilon) = U_0$. Similarly after a blow up it corresponds to a limit of solutions with boundary conditions $V(\varepsilon)= \varepsilon U_0^{-1}$. On the picture~\ref{fig:jacobi_conv} we can see this convergence numerically in the original phase portrait. 

\begin{figure}[ht!]  
\centering
\includegraphics[width=0.24\linewidth]{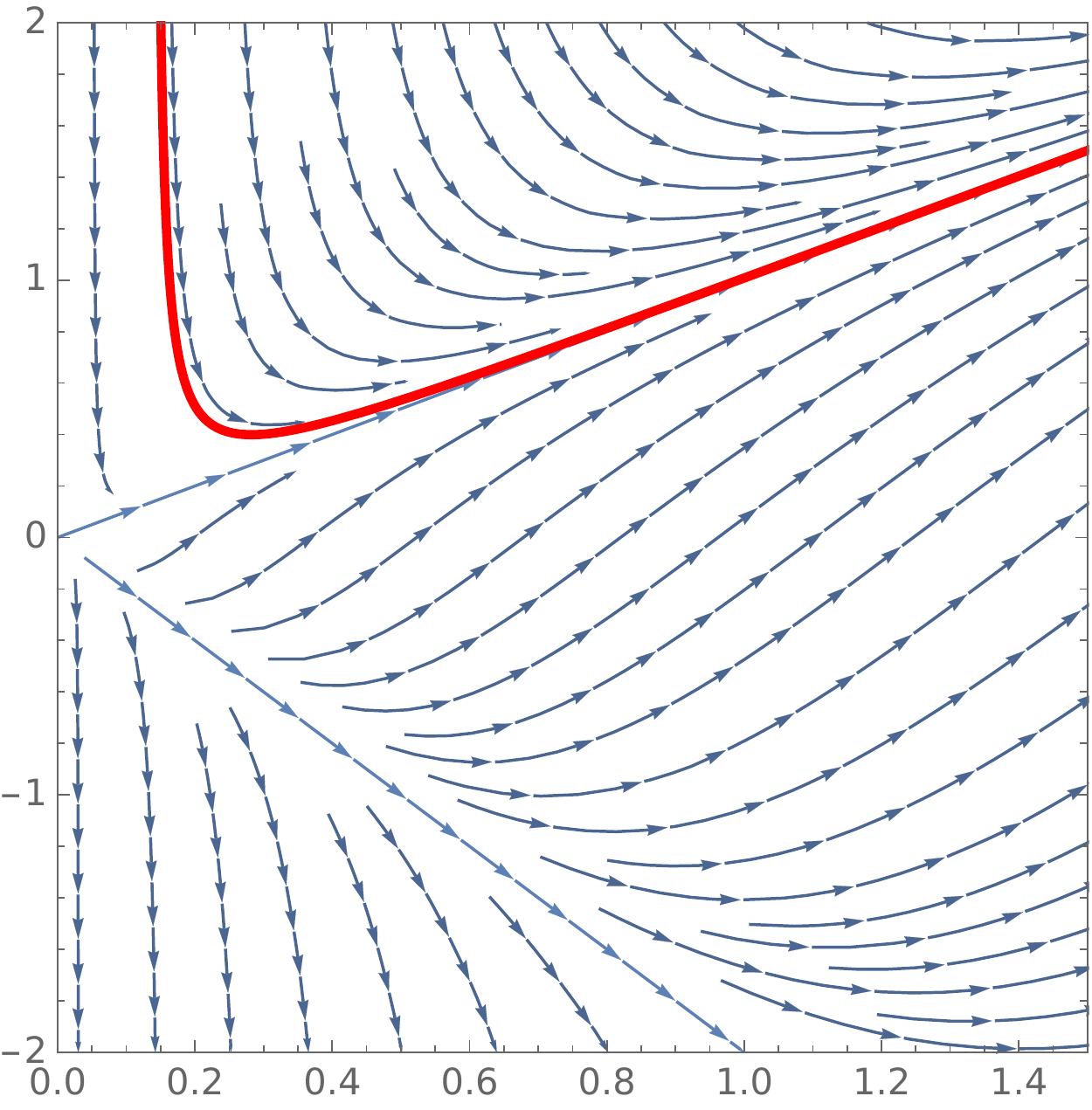} 
\includegraphics[width=0.24\linewidth]{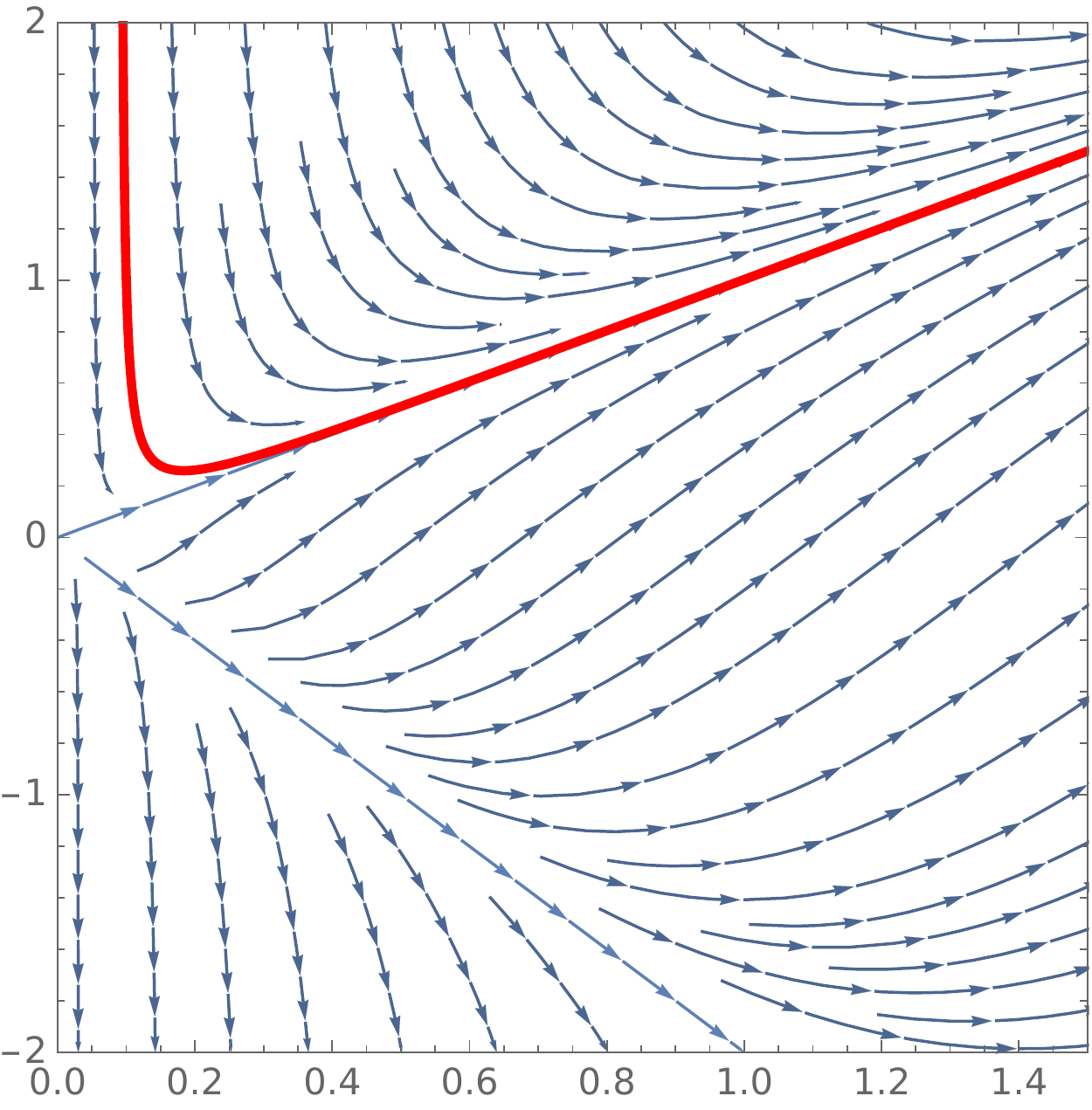}
\includegraphics[width=0.24\linewidth]{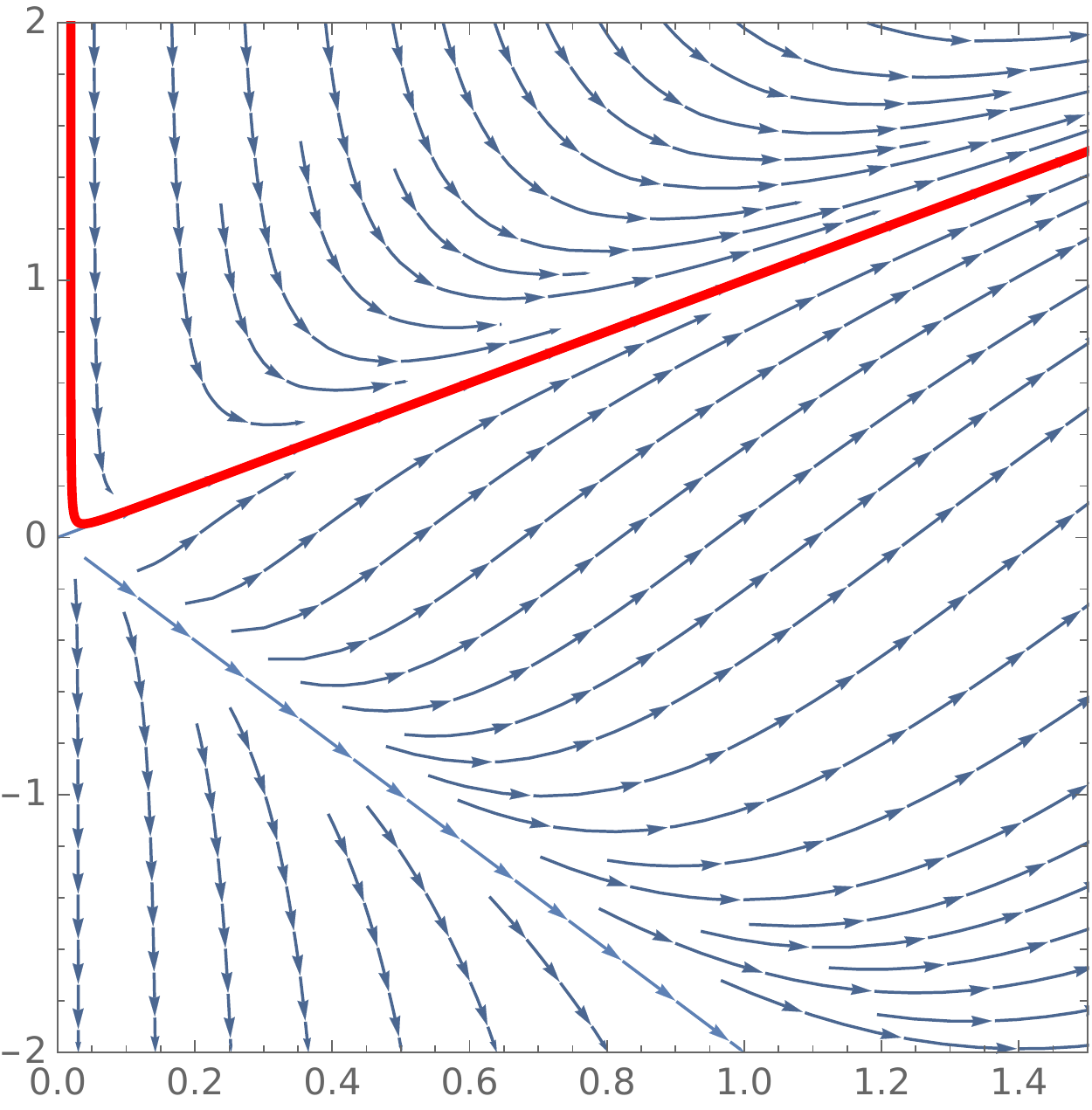}
\includegraphics[width=0.24\linewidth]{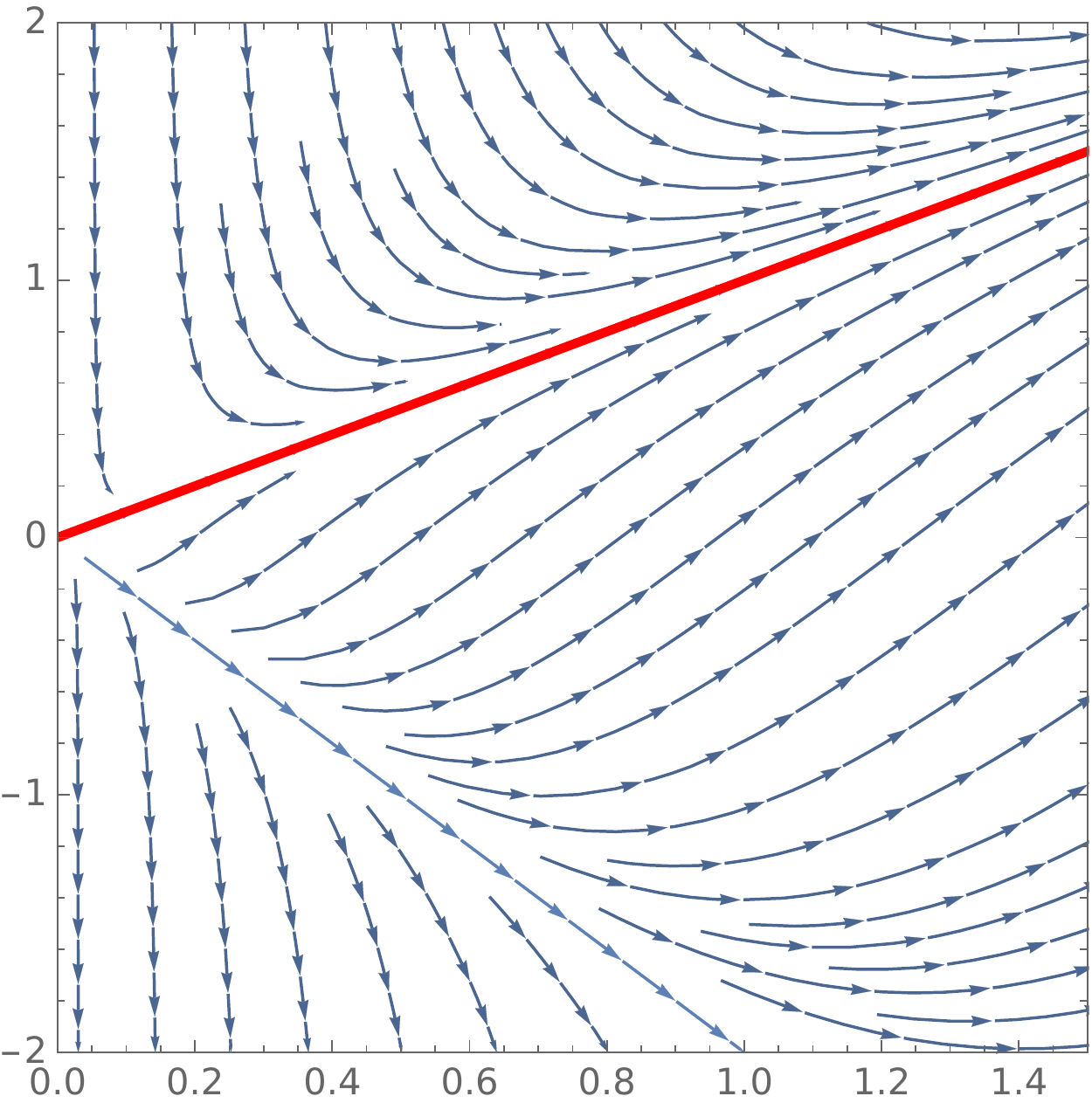}
\label{fig:jacobi_conv}
\caption{Point convergence of $\Lambda(\tau)$ to the Jacobi curve for $C=2$.}
\end{figure}

In the next sections our goal is to make all the ideas from this section rigorous. Our proves are mostly perturbative and we will first prove them for the constant matrix case and then expand it to the general case. Following the outline of this section we first prove an analogue of the Kneser criteria and identify the non-oscillating cases. Such a criteria is a necessary optimality condition on its own. Then using the general theory of ODEs with singular regular points we are going to characterize the jump of the Jacobi curve. Finally using the results from article~\cite{reid} we are going to characterize the first derivative of the Jacobi curve for $m=1,2$ and prove that conditions on the first derivative guarantee uniqueness of the solution of the extended Jacobi equation that characterizes $\cL_t$.

\section{Kneser oscillation criteria for a simple class of Hamiltonian systems}
\label{sec:kneser}
In the case $n=1$, the Kneser criteria gives sufficient conditions under which a second order ODE is oscillating or non-oscillating. This result is just a consequence of the Sturm comparison theorem and an explicit solution of an Euler-type linear equation~\cite{ode2}. 

Kneser criteria can be derived as a consequence of more general integral criteria and the modern theory of oscillation of ODE systems tends to generalize those. As a result, we were not able to find in the literature a similar simple point-criteria. Thus we would like to slightly generalize the Kneser criteria to a special class of Hamiltonian systems that include system~\eqref{eqnormal_jacobi}. Our main tool will be a straightforward consequence of Theorem 1 in~\cite{arnold2,agr_gamk_symp}, that can be seen as a generalization of the Strum comparison theorem
\begin{theorem}
\label{thm_sturm}
Let $I$ be an open interval and $A_i(\tau),B_i(\tau),C_i(\tau)$, $i=1,2$ quadratic matrices whose elements are differentiable on $I$. Assume that $B_i(t)$ and $C_i(t)$ are symmetric. We consider the corresponding Hamiltonians
$$
H_i = \begin{pmatrix}
C_i(\tau) & -A_i^T(\tau) \\
-A_i(\tau) & -B_i(\tau)
\end{pmatrix}
$$
If Hamiltonians $H_i(\tau)$ satisfy
$$
H_2(\tau) \geq H_1(\tau)
$$
then for any two trajectories $\Lambda_i(t)$ whose endpoints are transversal to $\Lambda$, we have the following inequality
$$
\Mi_{\Lambda}\Lambda_1(t) - n \leq \Mi_{\Lambda}\Lambda_2(t).
$$
\end{theorem}

We are going to use a simple direct corollary of that result
\begin{corollary}
\label{cor_sturm}
Let $I$ be an open interval and $A_i(\tau),B_i(\tau),C_i(\tau)$, $i=1,2$ quadratic matrices whose elements are differentiable on any compact subset $[a,b]\subset I$. Assume that $B_i(t)$ and $C_i(t)$ are symmetric and denote by $H_i(\tau)$ the corresponding Hamiltonians, s.t. $H_2(\tau) \geq H_1(\tau)$. Then we have the following implications:
\begin{enumerate}
\item If $\exists \Lambda\in L(n)$, s.t. $H_2(\tau)|_{\Lambda}\leq 0$ for all $\tau\in I$ and the second system is oscillating, then the first system is oscillating as well;
\item If $\exists \Lambda\in L(n)$, s.t. $H_1(\tau)|_{\Lambda}\geq 0$ for all $\tau\in I$ and the first system is oscillating, then the second system is oscillating as well;
\item If $\exists \Lambda\in L(n)$, s.t. $H_i(\tau)|_{\Lambda}\leq 0$ for all $\tau\in I$ and the first system is non-oscillating, then the second system is non-oscillating as well;
\item If $\exists \Lambda\in L(n)$, s.t. $H_i(\tau)|_{\Lambda}\geq 0$ for all $\tau\in I$ and the second system is non-oscillating, then the first system is oscillating as well;
\end{enumerate} 
\end{corollary}

\begin{proof}
The proof is just a corollary of Theorem~\ref{thm_sturm} and Theorem~\ref{thm:comparison}. For example, in the first case Theorem~\ref{thm:comparison} implies that
$$
\Mi_{\Lambda}\Lambda_2(\tau) \leq 0.
$$
The assumption that the corresponding system is oscillating means that the Maslov index of any solution is infinite. Then from the comparison Theorem~\ref{thm_sturm} we obtain
$$
\Mi_{\Lambda}\Lambda_1(\tau) \leq \Mi_{\Lambda}\Lambda_2(\tau) + n = -\infty.
$$
The remaining implications are proven in the same way. 
\end{proof}

The goal of this section is to prove the following result
\begin{theorem}
\label{thm_kneser}
Consider a Hamiltonian system
$$
\frac{d}{d\tau}\begin{pmatrix}
p\\
q
\end{pmatrix}
=
\begin{pmatrix}
A(\tau) & \frac{B(\tau)}{\tau^m}\\
C(\tau) & -A^T(\tau)\\
\end{pmatrix}
\begin{pmatrix}
p\\
q
\end{pmatrix},
$$
s.t. the following assumptions are satisfied 
\begin{enumerate}
\item $B(\tau)$ is a semi-definite smooth symmetric $n\times n$-matrix, s.t. $B(\tau)$ is sign-definite for $\tau>0$;
\item $C(\tau)$ is a smooth symmetric $n\times n$-matrix;
\item $A(\tau)$ an arbitrary smooth $n\times n$-matrix;
\end{enumerate}

Then the following statements are true:
\begin{enumerate}
\item Let $m =2$. If all the eigenvalues of the matrix $C(0)B(0)$ are strictly greater than $-1/4$, then the system is non-oscillating on $(0,\varepsilon)$. If at least one eigenvalue is smaller than $-1/4$, then the system is oscillating on the same interval;
\item If $0\leq m <2 $, then the system is not oscillating on $(0,\varepsilon)$;
\item Let $m >2$. If $C(0)$ is sign definite on the eigenspace that is transversal to the kernel of $B(0)$ with the same sign as $B(0)$, then the system is non-oscillating on $(0,\varepsilon)$. If $C(0)$ is not semi-definite on this subspace with the same sign as $B(0)$, then the system is oscillating on the same interval.

\end{enumerate}
\end{theorem}
We are pretty sure that this theorem can be derived as a consequence of some existing oscillation criteria for Hamiltonian systems, but we prefer to give here a simple geometric proof of the result using Theorem~\ref{thm_sturm}, that seems to be new.

As it can be seen from the statement the matrix $A$ plays no essential role. Similarly to a change of variables in the Section~\ref{sec:easy_case}, we make a time-dependent change of variables
$$
\begin{pmatrix}
p \\
q
\end{pmatrix}
\mapsto
\begin{pmatrix}
\Phi & 0 \\
0 & (\Phi^{-1})^T
\end{pmatrix}
\begin{pmatrix}
p \\
q
\end{pmatrix}
$$
where $\Phi(t)$ satisfies
$$
\left\{\begin{matrix}
\dot \Phi = A\Phi,\\ 
\Phi(0) = \id_n.
\end{matrix}\right.
$$
Then $\Phi(t)$ is the fundamental matrix of the corresponding linear equation and it is smooth. Therefore our change of variables is a non-degenerate symplectic change of variables and it does not change oscillatory properties of the Hamiltonian systems. 

Our Hamiltonian system now takes the form
$$
\frac{d}{d\tau}\begin{pmatrix}
p\\
q
\end{pmatrix}
=
\begin{pmatrix}
0 & \frac{\Phi^{-1}B(\Phi^{-1})^T}{\tau^m}\\
\Phi^T C \Phi & 0\
\end{pmatrix}
\begin{pmatrix}
p\\
q
\end{pmatrix},
$$
We can now simply redefine matrices $B$ and $C$. This implies that without any loss of generality, we can assume that $A(\tau)\equiv 0$.

In order to apply the comparison Theorem~\ref{thm_sturm}, we need a model example, which oscillating properties we understand very well. Such a model is given in the next lemma
\begin{lemma}
\label{lem_osc}
Consider a Hamiltonian system of the form
\begin{equation}
\label{eq:lin_ham}
\frac{d}{d\tau}\begin{pmatrix}
p\\
q
\end{pmatrix}
=
\begin{pmatrix}
0 & \frac{B}{\tau^2}\\
C & 0\
\end{pmatrix}
\begin{pmatrix}
p\\
q
\end{pmatrix},
\end{equation}
where $B$ and $C$ are constant symmetric matrices. This Hamiltonian system is oscillating on an interval $(0,\varepsilon)$ if and only if there exists at least one real eigenvalue $\lambda$ of the matrix $BC$, s.t. $\lambda<-1/4$.
\end{lemma}

This result is a consequence of the following theorem proven in~\cite{luca}. 
\begin{theorem}
\label{thm_luca}
A linear autonomous Hamiltonian system
 $$
\frac{d}{dt}\begin{pmatrix}
p\\
q
\end{pmatrix}
=
-JH
\begin{pmatrix}
p\\
q
\end{pmatrix},
$$
is oscillating on an unbounded interval if and only if the matrix $-JH$ has a purely imaginary eigenvalue. 
\end{theorem}

\begin{proof}[Proof of the Lemma~\ref{lem_osc}]
We do another symplectic transformation of the form
$$
\begin{pmatrix}
p \\
q
\end{pmatrix}
\mapsto
\begin{pmatrix}
\tau^{-1/2} & 0 \\
0 & \tau^{1/2}
\end{pmatrix}
\begin{pmatrix}
p \\
q
\end{pmatrix}
$$
The transformation is smooth for $\tau>0$ and therefore oscillating property is preserved. Our Hamiltonian system then becomes
$$
\tau\frac{d}{dt}\begin{pmatrix}
p \\
q
\end{pmatrix} = 
\begin{pmatrix}
\frac{1}{2}\id_n & B\\
C & -\frac{1}{2}\id_n
\end{pmatrix}
\begin{pmatrix}
p\\
q
\end{pmatrix}
= -JH\begin{pmatrix}
p \\
q
\end{pmatrix}
$$
Let us perform a change of time variable
$$
s = \ln \tau.
$$
Then we obtain a linear autonomous Hamiltonian system of the form
$$
\frac{d}{ds}\begin{pmatrix}
p \\
q
\end{pmatrix} = 
-JH
\begin{pmatrix}
p\\
q
\end{pmatrix}
$$
Note that the change of time that we have made, maps the bounded interval $(0,\varepsilon)$ to an unbounded one. So by the Theorem~\ref{thm_luca} it just remains to compute the eigenvalues of the matrix $-JH$, i.e. to solve
$$
\det(-JH - \lambda\id_{2n})=0.
$$
In this case the diagonal blocks are multipliers of the identity and hence commute with all the other blocks. Under this assumption it is easy to show that
$$
\det(-JH - \lambda\id_{2n})=\det\left( \left( \lambda^2 -\frac{1}{4} \right)\id_n -BC\right).
$$
If a matrix $-JH$ has a pair of purely complex eigenvalues $\lambda= \pm i b$, we obtain that the matrix $BC$ has an eigenvalue
$-b^2-1/4 < -1/4$. It is obvious that the converse holds as well.  So the result follows from Theorem~\ref{thm_luca}.
\end{proof}

Finally we need the following fact proven in~\cite{horn}:
\begin{theorem}
\label{thm_spectrum}
Let $B,C$ be two constant symmetric matrices, s.t. one of them is semidefinite. Then the spectrum of $BC$ is real.
\end{theorem}

\begin{proof}[Proof of the Theorem~\ref{thm_kneser}]
We assume that $B(\tau)< 0$ for sufficiently small $\tau>0$. The case $B(\tau)> 0$ is proven in a similar way. In this case the corresponding Hamiltonian $H$ is positive semidefinite on the horizontal plane $\Sigma$ (the $q$-plane).

1) Let us start with the case $m = 2$. We know by the previous theorem that all the eigenvalues of $B(0)C(0)$ are real, and let us assume first that the minimum one is strictly less then $-1/4$. We define
\begin{align*}
B_1(\tau) & = B(0)+\varepsilon\id_n,  & B_2(\tau) &= B(\tau),\\
C_1(\tau) & = C(0)-\varepsilon\id_n,  & C_2(\tau) &= C(\tau).
\end{align*}

Since we will apply Theorem~\ref{lem_osc} several times it is worth to note that for the Hamiltonians that we study, we have
$$
H_1(\tau) \leq H_2(\tau), \quad \forall \tau >0 \quad \iff \quad \begin{array}{l}
B_1(\tau) - B_2(\tau) \geq 0,\\
C_2(\tau) - C_1(\tau) \geq 0,
\end{array} \quad \forall \tau>0. 
$$

By assumption of the theorem and Lemma~\ref{lem_osc} the Hamiltonian system \eqref{eq:lin_ham} with $B=B(0)$ and $C=C(0)$ is oscillating. This implies that a system of the form with $B=B(0)+\varepsilon\id_n$ and $C=C(0)-\varepsilon\id_n$ must be oscillating as well. Indeed, eigenvalues of $BC$ are solutions of the characteristic equation whose coefficients depend continuously on the coefficients of matrices $B,C$. Therefore a small perturbation of matrices produces a small change in the eigenvalues of $BC$. But we have chosen such a perturbation in a way, that according to Theorem~\ref{thm_spectrum}, the spectrum of $BC$ remains real. So all the eigenvalues shift on the real axis, and if we choose $\varepsilon>0$ small enough, the minimum eigenvalue of the perturbed matrix will stay strictly smaller then $-1/4$ and the corresponding Hamiltonian system \eqref{eq:lin_ham} stays oscillating by Lemma~\ref{lem_osc}.

So we can use the implication 2) of Corollary~\ref{cor_sturm}. By smoothness assumption, indeed, for sufficiently small times $H_1(\tau) \leq H_2(\tau)$.  

2) The non-oscillating case for $m=2$ is proven using exactly the same argument and matrices
\begin{align*}
B_1(\tau) & = B(\tau),  & B_2(\tau) &= B(0)-\varepsilon\id_n   ,\\
C_1(\tau) & = C(\tau)   & C_2(\tau) &= C(0)+\varepsilon\id_n.
\end{align*}

3) The case $0\leq m<2$ is now just a consequence of what we have proven so far. Indeed, we can consider the Hamiltonian as having a singularity with $m=2$ and with a new matrix $\hat{B}(\tau)=\tau^{2-m}B(\tau)$. Then $\hat{B}(0)C(0) = 0$ and all the eigenvalues are zero. Hence the system is not oscillating.

4) Let us now assume that $m>2$. We first we apply a symplectic transform
$$
\begin{pmatrix}
R & 0 \\
0 & R^T
\end{pmatrix},
$$
where $R\in \gSO(n)$. Then in the new coordinates $B(\tau)$ and $C(\tau)$ will be replaced by the same matrices conjugated with $R$. Let us choose $R$, s.t. in the new coordinates $B(0)$ is diagonal. 

We take 
\begin{align*}
B_1(\tau) & = -k\tau^{m-2}\begin{pmatrix}
\id_l & 0 \\
0 & 0
\end{pmatrix}  & B_2(\tau) &= B(\tau),\\
C_1(\tau) & = C(\tau),  & C_2(\tau) &=  C(\tau),
\end{align*}
where $k$ is some constant. If $C(0)$ is non-negative semi-definite on the eigenspace of $B(0)$ that corresponds to the non-zero eigenvalue, then by taking $k$ large enough, we find that the system one is oscillating by Lemma~\ref{lem_osc}. We also have that $B_1(\tau) - B_2(\tau) $ is non-negative for sufficiently small $\tau>0$. Thus we can use the implication 2) of Corollary~\ref{cor_sturm} to deduce that the second system is going to be oscillating as well.

If $C(0)$ is negative definite on the eigenspace of $B(0)$, we repeat the proof with exactly the same $B_1$, $C_1$. In this case we know from what we have proven already, that the system one is not oscillating. Then the result follows from the implication 3) from Corollary~\ref{cor_sturm} with $\Lambda=\Pi$ being the vertical plane (the $p$-plane).

\end{proof}

Let us apply the theorem to our case. We are not oscillating for $0\leq m < 2$ and for $m=2$ whenever the eigenvalues of the
$$
B(0)C(0) = \begin{pmatrix}
\frac{1}{b_2} & 0 \\
0 & 0
\end{pmatrix}\begin{pmatrix}
\sigma(e_1(0),\dot{e}_1(0)) & 0 \\
0 & \sigma(e_2(0),\dot{e}_2(0))
\end{pmatrix} = 
\begin{pmatrix}
\sigma(X(0),\dot{X}(0))/b_2 & 0 \\
0 & 0
\end{pmatrix}
$$
are bigger then $-1/4$, i.e. whenever the only non-zero element above is bigger then $-1/4$. If $m>2$ then we must have $\sigma(X(0),\dot X(0))<0$.

So from now on, we assume that our Hamiltonian system is non-oscillating, ensuring that the right limit $\cL_{0+}$ exists.

\section{Computing the jump when $n\geq 2$}
\label{sec:jumpg2}
We are now ready to make the first step and compute the jump of the Jacobi curve. We will give three similar but separate proofs for $m =1$, $m=2$ and $m\geq 3$. For $m=1$ we will do this using the general theory of linear ODEs with regular singular points. For $m=2$ and $m\geq 3$ the strategy of the proof is going to be very similar to the proof of the Kneser theorem in the previous section. Namely we first look at some model examples and then we use the comparison theory of Riccati equations, to obtain the result in the most general case. 

\subsection{Jump for $m=1$}
\label{subsec:jump1}
Let $\lambda=(p,q)$. We rewrite the system \eqref{eqnormal_jacobi} in the following form
\begin{equation}
\label{eqcase1}
\dot{\lambda} = \left( \frac{H_{-1}}{\tau} + H(\tau) \right)\lambda,
\end{equation}
where $H(\tau)$ is an analytic matrix function and as can be easily seen
$$
H_{-1} = \begin{pmatrix}
0 & 0 & \frac{1}{b_1} & 0 \\
0 & 0 & 0 & 0 \\
0 & 0 & 0 & 0 \\
0 & 0 & 0 & 0
\end{pmatrix}.
$$
This matrix up to a reordering of coordinates is in its Jordan normal form, and all of its eigenvalues are zero. Therefore by a well known theorem~\cite{ode}, the fundamental matrix $\Phi(\tau)$ of the system \eqref{eqcase1} can be written as
$$
\Phi(\tau) = P(\tau)\tau^{H_{-1}},
$$
where $P(\tau)$ is an analytic matrix function with $P(0) = \id_{2n}$. A power series expansion can be obtained by plugging this solution into \eqref{eqcase1} and expanding all the analytic functions into their Taylor series. It is easy to check that
$$
\tau^{H_{-1}} = \begin{pmatrix}
1 & 0 & \frac{\ln\tau}{b_1} & 0 \\
0 & 1 & 0 & 0 \\
0 & 0 & 1 & 0 \\
0 & 0 & 0 & 1 
\end{pmatrix}
$$
Let $\cL_0$ be the $\cL$-derivative at moment of time $\tau=0$. The flow of the Hamiltonian system $\Phi(\tau)$ induces a flow on the Lagrangian Grassmanian $L(2)$ that we denote using the same symbol. As we have discussed previously the Jacobi curve $\cL_\tau$ is going to be a pointwise limit of the solutions of the Jacobi DE on $L(2)$ with boundary conditions $\Lambda(\varepsilon)=\cL_0$. Since we know explicitly the flow, we can write the solution of this boundary problem as
\begin{equation}
\label{eqjacobi_limit}
\cL_\tau = \lim_{\varepsilon\to 0+} \Phi(\tau)\Phi^{-1}(\varepsilon) \cL_0.
\end{equation}
We note that $\Phi(\tau)$ is smooth and invertible for $\tau>0$. Therefore we can exchange the limit with $\Phi(\tau)$, and we just need to compute the limit of $\Phi^{-1}(\varepsilon)\cL_0$. To do this we use a concrete representation of Lagrangian planes as span of a couple of vectors like in Section~\ref{sec:sympl}
$$
\Phi^{-1}(\varepsilon)\cL_0 = \begin{bmatrix}
\lambda_1(\varepsilon) & \lambda_2(\varepsilon)
\end{bmatrix}.
$$

Let us find the limits of $\lambda_i(\varepsilon)$ as $\varepsilon\to 0$. We have 
$$
(\varepsilon^{H_{-1}})^{-1} =\begin{pmatrix}
1 & 0 & -\frac{\ln\varepsilon}{b_1} & 0 \\
0 & 1 & 0 & 0 \\
0 & 0 & 1 & 0 \\
0 & 0 & 0 & 1 
\end{pmatrix}.
$$
Assume that $X(0) \in \cL_0$. Then as we have seen in the Example~\ref{ex2} we can assume
\begin{equation}
\label{eqjacobicase1}
\cL_0 = \begin{bmatrix}
1 & 0\\
0 & y_2\\
0 & 0\\
0 & w_2
\end{bmatrix}
\end{equation}

Then since $P(0) = \id_{2n}$, we obtain
$$
\lim_{\varepsilon \to 0+}\Phi^{-1}(\varepsilon)\cL_0 = \lim_{\varepsilon\to 0+} (\varepsilon^{H_{-1}})^{-1}P^{-1}(\varepsilon) \begin{bmatrix}
1 & 0 \\
0 & y_2 \\
0 & 0 \\
0 & w_2
\end{bmatrix} = \begin{bmatrix}
1 & 0 \\
0 & y_2 \\
0 & 0 \\
0 & w_2
\end{bmatrix},
$$
because $(\varepsilon^{H_{-1}})^{-1}$ acts as the identity on $\cL_0$. For the same reason
$$
\lim_{\tau \to 0+} \Lambda(\tau) = \begin{bmatrix}
1 & 0 \\
0 & y_2 \\
0 & 0 \\
0 & w_2
\end{bmatrix}
$$
and so $\cL_{0+} = \cL_0$ and the Jacobi curve is actually continuous. 

If $X(0)\notin\cL_0$, then again from Example~\ref{ex2} we know that we can take
\begin{equation}
\label{eqjacobicase2}
\cL_0 = 
\begin{bmatrix}
\lambda_1(\varepsilon) & \lambda_2(\varepsilon)
\end{bmatrix}
=
\begin{bmatrix}
x_1 & x_2\\
y_1 & y_2\\
1 & 0\\
w_1 & w_2
\end{bmatrix}
\end{equation}
Similarly to the previous case we find that
$$
\lim_{\varepsilon \to 0+}\lambda_2(\varepsilon) = \begin{pmatrix}
x_2\\
y_2\\
0 \\
z_2
\end{pmatrix}.
$$
Let us see what happens to the limit of the first vector. We have
$$
\lim_{\varepsilon \to 0+}\lambda_1(\varepsilon) = \lim_{\varepsilon\to 0+} \Phi^{-1}(\varepsilon)
\begin{pmatrix}
x_1\\
y_1\\
1\\
w_1
\end{pmatrix} = \lim_{\varepsilon\to 0+}
\begin{pmatrix}
x_1-\dfrac{\ln\varepsilon}{b_1}\\
y_1\\
1\\
w_1
\end{pmatrix}
$$
which is equal to infinity. As we have said before a representation of a Lagrangian plane as a span of two vectors is not unique. We can scale them as we want as we take the limit. So we take
$$
\lim_{\varepsilon\to 0+} -\frac{b_1}{\ln \varepsilon}\lambda_1(\varepsilon)  = 
\begin{pmatrix}
1\\
0\\
0\\
0
\end{pmatrix}.
$$
So
$$
\lim_{\varepsilon \to 0+}\Phi^{-1}(\varepsilon)\cL_0 = \begin{bmatrix}
1 & x_2\\
0 & y_2 \\
0 & 0 \\
0 & w_2
\end{bmatrix}.
$$
Then as before we find that
$$
\lim_{\tau\to 0+}\Lambda(\tau) = \begin{bmatrix}
1 & x_2\\
0 & y_2 \\
0 & 0 \\
0 & w_2
\end{bmatrix}
=
\begin{bmatrix}
1 & 0\\
0 & y_2 \\
0 & 0 \\
0 & w_2
\end{bmatrix}.
$$

So summarizing everything we have done in a more invariant manner the jump can be computed as follows. Given $\cL_0$, the new $\cL$-derivative $\cL_{0+}$ is going to be a direct sum of $\cL_0\cap X(0)^\angle$ and $X(0)$. But this is by definition $\cL_0^{X(0)}$. The goal of the following subsections is to prove the same for $m\geq 2$.

\subsection{Model examples for $m=2$}
\label{subsec:model2}

For $m\geq 2$ we proceed in a different way. One can reduce by a change of variables the system \eqref{eqnormal_jacobi} to a system with a regular singular point at $\tau=0$. Thus all the methods from the theory of linear systems of ODE's can be used. But these techniques work well under some non-resonance conditions. In our case we can use techniques from Hamiltonian dynamics to arrive at the results even in the presence of resonances. First we prove the result for some model problems similarly as we have done in the case $m = 1$. Then we apply Riccati comparison theorems, to prove the general result.

For $m = 2$ we choose the following Hamiltonian systems as our models
\begin{equation}
\label{eqmodel2}
\frac{d}{dt}\begin{pmatrix}
p\\
q
\end{pmatrix}
=
\begin{pmatrix}
0 & 0 & \frac{b_{11}}{\tau^2} & 0 \\
0 & 0 & 0 & b_{22} \\
c_{11} & 0 & 0 & 0 \\
0 & c_{22} & 0 & 0
\end{pmatrix}
\begin{pmatrix}
p\\
q
\end{pmatrix},
\end{equation}
where $b_{ii},c_{ii}$ are constants, $b_{11}\neq 0$ (or else there is no singularity) and $c_{11}\neq 0$ (because we have $\sigma(X(0),\dot X(0))\neq 0$). First of all we notice that this system splits into two invariant sub-systems
\begin{align}
\frac{d}{dt}\begin{pmatrix}
p_1\\
q_1
\end{pmatrix}
&=
\begin{pmatrix}
0  & \frac{b_{11}}{\tau^2} \\
c_{11} & 0 
\end{pmatrix}
\begin{pmatrix}
p_1\\
q_1
\end{pmatrix}, \label{eqalpha21}\\
\frac{d}{dt}\begin{pmatrix}
p_2\\
q_2
\end{pmatrix}
&=
\begin{pmatrix}
0  & b_{22} \\
c_{22} & 0 
\end{pmatrix}
\begin{pmatrix}
p_2\\
q_2
\end{pmatrix}.\label{eqalpha22}
\end{align}
We denote by $\Phi_i(t)$ the corresponding fundamental matrices. Without any loss of generality, we can assume that $\Phi_2(0) = \id_2$. In order to find $\Phi_1$, we do a symplectic change of variables
$$
\begin{pmatrix}
\tilde{p}_1\\
\tilde{q}_1
\end{pmatrix}
=
\begin{pmatrix}
\tau^{1/2} & 0 \\
0 & \tau^{-1/2}
\end{pmatrix}
\begin{pmatrix}
 p_1\\
 q_1
\end{pmatrix}.
$$ 
Then the first system is transformed to
$$
\frac{d}{dt}\begin{pmatrix}
\tilde p_1\\
\tilde q_1
\end{pmatrix}
=
\frac{1}{\tau}\begin{pmatrix}
1/2 & b_{11}\\
c_{11} & -1/2
\end{pmatrix}
\begin{pmatrix}
\tilde p_1\\
\tilde q_1
\end{pmatrix} = 
\frac{Y}{\tau}
\begin{pmatrix}
\tilde p_1\\
\tilde q_1
\end{pmatrix}
$$
which is a simple linear system with a regular singular point. Therefore the fundamental solution $\Phi_1$ is the following matrix function
$$
\Phi_1(\tau) = \begin{pmatrix}
\tau^{-1/2} & 0 \\
0 & \tau^{1/2}
\end{pmatrix}\tau^Y
$$
or in a more detailed form
$$
\Phi_1(\tau) = \begin{pmatrix}
\dfrac{\tau^{\frac{-1-\Delta}{2}}(-1+\Delta+(1+\Delta)\tau^\Delta)	}{2\Delta} & \dfrac{b_{11}\tau^{\frac{-1-\Delta}{2}}(-1+\tau^\Delta)}{\Delta} \\
\dfrac{c_{11}\tau^{\frac{1-\Delta}{2}}(-1+\tau^\Delta)}{\Delta} & \dfrac{\tau^{\frac{1-\Delta}{2}}(1+\Delta+(-1+\Delta)\tau^\Delta)	}{2\Delta}
\end{pmatrix},
$$
where $\Delta = \sqrt{1+4b_{11}c_{11}}$. Under the non-oscillation assumption we have $\Delta>0$. It is easy to check that $\det \Phi_1(\tau) = \det \Phi_2(\tau) \equiv 1$, so the inverse matrix of $\Phi(\tau)$ can be computed easily. 

The Jacobi curve is given by \eqref{eqjacobi_limit} and as in the previous subsection the fundamental matrix $\Phi(\tau)$ is smooth for $\tau>0$, so we can exchange it with the limit. So first of all we need to find the limit
$$
\lim_{\varepsilon\to 0+} \Phi^{-1}(\varepsilon)\cL_0.
$$
As in the previous case we are going to separate two situations: when $X(0) \in \cL_0$ and when $X(0) \notin \cL_0$. 

As before if $X(0) \in \cL_0$ we can assume that $\cL_0$ is given by \eqref{eqjacobicase1}. Then the first and the second vector lie in its own invariant subspace. For example since $\Phi_2(0) = \id_2$, we immediately get that
$$
\lim_{\tau\to 0+} \Phi(\tau)\lim_{\varepsilon\to 0+} \Phi^{-1}(\varepsilon)\begin{pmatrix}
0\\
y_2\\
0\\
w_2
\end{pmatrix} =\begin{pmatrix}
0\\
y_2\\
0\\
w_2
\end{pmatrix}.
$$
Let us see what happens to the first vector. We have
$$
\Phi^{-1}(\varepsilon)\begin{pmatrix}
1\\
0\\
0\\
0
\end{pmatrix}
=
\begin{pmatrix}
\dfrac{\varepsilon^{\frac{1-\Delta}{2}}(1+\Delta+(-1+\Delta)\varepsilon^\Delta)	}{2\Delta}\\
0\\
-\dfrac{c_{11}\varepsilon^{\frac{1-\Delta}{2}}(-1+\varepsilon^\Delta)}{\Delta}\\
0
\end{pmatrix}.
$$
Then we find that
$$
\lim_{\varepsilon\to 0+} \left[\Phi^{-1}(\varepsilon)\begin{pmatrix}
1\\
0\\
0\\
0
\end{pmatrix}\right] = 
\lim_{\varepsilon\to 0+} \left[\varepsilon^{-\frac{1-\Delta}{2}}\Phi^{-1}(\varepsilon)\begin{pmatrix}
1\\
0\\
0\\
0
\end{pmatrix}\right] = 
\begin{bmatrix}
\frac{1+\Delta}{2\Delta} \\
0\\
\frac{c_{11}}{\Delta}\\
0
\end{bmatrix}\in \lim_{\varepsilon\to 0+}\Phi^{-1}(\varepsilon)\cL_0
$$
and
\begin{equation}
\label{eqn1}
\lim_{\tau\to 0+}\left[\tau^{\frac{1+\Delta}{2}}\Phi(\tau)\begin{pmatrix}
\frac{1+\Delta}{2\Delta} \\
0\\
\frac{c_{11}}{\Delta}\\
0
\end{pmatrix}\right] = \begin{bmatrix}
1\\
0\\
0\\
0
\end{bmatrix},
\end{equation}
which means that the limit is up to a constant the vector $X(0)$. Thus in this case $\cL_{0+}=\cL_0$ and the Jacobi curve is actually continuous as expected. 

We now look at the situation when $X(0) \notin \cL_0$. Then $\cL_0$ can be assumed to be of the form \eqref{eqjacobicase2}. We consider the vectors $\lambda_i(\varepsilon)$ and their projections onto the first invariant subspace
\begin{equation}
\label{eqn1n2}
n_1(\varepsilon) = \Phi_1^{-1}(\varepsilon)\begin{pmatrix}
x_1 \\
1 
\end{pmatrix},\qquad
n_2(\varepsilon) = \Phi_1^{-1}(\varepsilon)\begin{pmatrix}
x_2 \\
0
\end{pmatrix}.
\end{equation}
We have 
$$
n_1(\varepsilon) = \varepsilon^{-\frac{1+\Delta}{2}} \begin{pmatrix}
\dfrac{-2b_{11}(-1+\varepsilon^\Delta)+\varepsilon(1+\Delta +(-1+\Delta)\varepsilon^\Delta)x_1}{2\Delta}\\
\dfrac{-1+\Delta+2c_{11}\varepsilon x_1 +\varepsilon^\Delta(1-2c_{11}\varepsilon x_1 + \Delta)}{2\Delta}
\end{pmatrix}
$$
So it is clear that
$$
\lim_{\varepsilon\to 0+} \varepsilon^{\frac{1+\Delta}{2}} \lambda_1(\varepsilon) = \begin{pmatrix}
\dfrac{b_{11}}{\Delta}\\
0
\\
\dfrac{-1+\Delta}{2\Delta}\\
0
\end{pmatrix}.
$$
Then the formula \eqref{eqn1} proves that $X(0) \in \cL_{0+}$. Now we need to find an independent from $X(0)$ limit vector that would lie in $\cL_{0+}$. 

Writing down $n_2(\varepsilon)$ we get
$$
n_2(\varepsilon) = \varepsilon^{\frac{1-\Delta}{2}}\begin{pmatrix}
\dfrac{(1+\Delta+\varepsilon^\Delta(-1+\Delta))x_2}{2\Delta}\\
\dfrac{c_{11}(1-\varepsilon^\Delta)x_2}{\Delta}
\end{pmatrix}.
$$
If $0< \Delta < 1$ or $x_2 = 0$, then it is clear that
$$
\lim_{\varepsilon\to 0+} \lambda_2(\varepsilon) = \begin{pmatrix}
0\\
y_2 \\
0\\
w_2
\end{pmatrix},
$$
where we have used that $\Phi_2(0) = \id_2$. For the same reason the very same vector is going to lie in $\cL_{0+}$ and the result follows.

If $\Delta = 1$, then either $b_{11}=0$ or $c_{11} = 0$. Since we have excluded these possibilities it only remains to see what happens, when $\Delta>1$ and $x_2 \neq 0$. 

We can see that the expressions for $\lambda_i(\varepsilon)$ are just sums of power series of $\varepsilon$. Therefore it is convenient to introduce the following notation
$$
a(\varepsilon) = b(\varepsilon) \mod \varepsilon^{>0}
$$
which means that $a(\varepsilon)$ and $b(\varepsilon)$ agree modulo terms of positive degree in $\varepsilon$. Then $\cL_{0+} = \cL_0^{X(0)}$ follows from the following lemma
\begin{lemma}
Let $\Delta>1$ and $x_2\neq 0$. Then there exist constants $c_0,c_1,...,c_{l-1}$, s.t.
$$
\lambda_2(\varepsilon) - x_2\varepsilon \lambda_1(\varepsilon)\sum_{i=0}^{l-1} c_i \varepsilon^i = \begin{pmatrix}
k_1 \varepsilon^{\frac{2l+1-\Delta}{2}} \\
y_2 \\
k_2 \varepsilon^{\frac{2l+1-\Delta}{2}}\\
w_2
\end{pmatrix} \mod \varepsilon^{>0},
$$
where $k_i$ are some constants. 
\end{lemma}
Indeed, the vector on the left hand side is a linear span of $\lambda_1(\varepsilon)$ and $\lambda_2(\varepsilon)$. We can choose $l$ sufficiently big so that $2l+1-\Delta>0$. Then in the limit we obtain a vector $\begin{pmatrix}
0 & y_2 & 0 & w_2
\end{pmatrix}^T$, which lies in the second invariant subspace where there is no singularity at all.

\begin{proof}[Proof of the lemma]
We denote by $\lambda(\varepsilon)$ the vector on the left. It is easy to see why the second and fourth components of $\lambda(\varepsilon)$ have this form. It follows from the fact that $\Phi_2(\varepsilon)$ is an analytic matrix function with $\Phi_2(0) = \id_2$.

So it is enough to look on the projection of $\lambda(\varepsilon)$ to the singular invariant subspace. We can write
\begin{align*}
n_1(\varepsilon) &= \frac{\varepsilon^{-\frac{1+\Delta}{2}}}{\Delta}\begin{pmatrix}
b_{11}\\
\frac{-1+\Delta}{2} 
\end{pmatrix}+
\frac{\varepsilon^{\frac{1-\Delta}{2}}x_1}{\Delta}\begin{pmatrix}
\frac{1+\Delta}{2}\\
c_{11}
\end{pmatrix}\mod \varepsilon^{>0},\\
n_2(\varepsilon) &= 
\frac{\varepsilon^{\frac{1-\Delta}{2}}x_2}{\Delta}\begin{pmatrix}
\frac{1+\Delta}{2}\\
c_{11}
\end{pmatrix}\mod \varepsilon^{>0}.
\end{align*}
Let us denote 
$$
\alpha(\varepsilon) = \frac{\varepsilon^{\frac{1-\Delta}{2}}}{\Delta}\begin{pmatrix}
b_{11}\\
\frac{-1+\Delta}{2} 
\end{pmatrix}
$$
From here we see that
$$
x_2\varepsilon n_1(\varepsilon) = x_2 \alpha(\varepsilon)+
x_1\varepsilon n_2(\varepsilon)\mod \varepsilon^{>0}.
$$
We then find an expression for the projection of $\lambda(\varepsilon)$:
$$
n_2(\varepsilon) -x_2\varepsilon n_1(\varepsilon)\sum_{i=0}^{l-1} c_i\varepsilon^i=  
$$
$$
=n_2(\varepsilon) - x_2\alpha(\varepsilon)c_0 + \sum_{i=1}^{l-2} \left( x_2 \alpha(\varepsilon) c_{i} - x_1 n_2(\varepsilon)c_{i-1} \right) - \varepsilon^{l}x_1n_2(\varepsilon)c_{l-1}     \mod \varepsilon^{>0}.
$$
So it is enough to choose $c_i$ to be s.t. they solve
$$
n_2(\varepsilon) - x_2\alpha(\varepsilon)c_0 = 0 \mod \varepsilon^{>0},
$$
$$
x_2\alpha(\varepsilon)c_{i+1} - x_1n_2(\varepsilon)c_i  = 0 \mod \varepsilon^{>0}.
$$
The first equality is satisfied, if 
$$
c_0 = \frac{1+\Delta}{2b_{11}}, \qquad \text{(recall that } \Delta = \sqrt{1+4b_{11}c_{11}}).
$$
But then we can obtain an expression for $n_2(\varepsilon)$ from the first equation and plug it into the second one. We get
$$
\alpha(\varepsilon)x_2\left(c_{i+1} - x_1 c_0 c_i\right) = 0.
$$
So we simply choose recursively $c_{i+1} = x_1c_ic_0$.
\end{proof}

\subsection{Model examples for $m>2$}
\label{subsec:modelg2}
We now consider the same model as \eqref{eqmodel2} but with singularity of order $m>2$. Recall that $B_{ii},C_{ii}$ are constants and $b_{11}$, $c_{11}$ are non zero. It is convenient to define $m = 2 +\beta$. Again we have two invariant subsystems and equation \eqref{eqalpha21} has the form
\begin{align*}
\dot{p}_1 &= \frac{b_{11}}{\tau^{2+\beta}} q_1,\\
\dot{q}_1 &= c_{11} p_1.
\end{align*}
We differentiate the second equation to obtain
$$
\ddot q_{1} - \frac{b_{11}c_{11}}{\tau^{2+\beta}}q_1 = 0.
$$
If we introduce a new independent variable 
$$
y(\tau) = \frac{q_1(\tau)}{\sqrt{\tau}}
$$
and a new dependent variable
$$
s(\tau) = \frac{2\sqrt{b_{11}c_{11}}}{\beta}\tau^{\frac{-\beta}{2}},
$$
we obtain a modified Bessel equation
$$
s^2\frac{d^2y}{ds^2}+s\frac{dy}{ds}-\left( s^2+\frac{1}{\beta^2} \right)y =0.
$$
Two independent solutions of this equation are given by two modified Bessel functions $I_{\beta^{-1}}(s)$, $K_{\beta^{-1}}(s)$~\cite{bessel}. Therefore the fundamental matrix $\Phi_1(\tau)$ is given by
$$
\Phi_1(\tau) = \begin{pmatrix}
\frac{1}{c_{11}}\frac{d}{d\tau}\sqrt{\tau}I_{\beta^{-1}}(s(\tau)) & \frac{1}{c_{11}}\frac{d}{d\tau}\sqrt{\tau}K_{\beta^{-1}}(s(\tau)) \\
\sqrt{\tau}I_{\beta^{-1}}(s(\tau)) & \sqrt{\tau}K_{\beta^{-1}}(s(\tau))
\end{pmatrix}
$$
We can simplify considerably the first row using the following formulas for the derivatives of modified Bessel functions~\cite{bessel}
\begin{align*}
I'_{a}(x) &= \frac{a}{x}I_a(x)+I_{a+1}(x),\\
K'_{a}(x) &= \frac{a}{x}K_a(x)-K_{a+1}(x).
\end{align*}
After some simplifications we find that
$$
\Phi_1(\tau) = \begin{pmatrix}
-\sqrt{\frac{b_{11}}{c_{11}}}\tau^{-\frac{1+\beta}{2}}I_{\beta^{-1}+1}(s(\tau))
 & \sqrt{\frac{b_{11}}{c_{11}}}\tau^{-\frac{1+\beta}{2}}K_{\beta^{-1}+1}(s(\tau)) \\
\sqrt{\tau}I_{\beta^{-1}}(s(\tau)) & \sqrt{\tau}K_{\beta^{-1}}(s(\tau))
\end{pmatrix}.
$$

Since $\beta>0$, as $\tau \to 0+$ we get $s(\tau) \to +\infty$. Therefore we need an asymptotic expansion of modified Bessel functions as the argument goes to $+\infty$:
\begin{align*}
I_{a}(x) &\sim \sqrt{\frac{1}{2\pi x}}e^x, \qquad x\to+\infty,\\
K_{a}(x) &\sim \sqrt{\frac{\pi}{2 x}}e^{-x}, \qquad x\to+\infty.
\end{align*}
In particular we see that the limit does not depend on the parameter, and therefore for any real $a,b$ we have
\begin{align*}
&I_{a}(x) \to +\infty, & &\frac{I_a(x)}{I_b(x)} \to 1,\\
&K_{a}(x) \to 0, & &\frac{K_a(x)}{K_b(x)} \to 1,
\end{align*}
as $x\to +\infty$.

The matrix $\Phi_1(\tau)$ is invertible and smooth for $\tau>0$. From the explicit form of the equation it follows that determinant  of $\Phi_1(\tau)$ is constant. Using the asymptotics above we can the find that it is actually equal to $-\beta/(2c_{11})$. The very same asymptotics and an argument similar to the one for $m=1,2$ implies that if $X(0) \in \cL_0$, then $\cL_{0+}=\cL_0$. So we assume that $X(0) \notin \cL_0$ and consequently that $\cL_0$ is given by \eqref{eqjacobicase2}.

If $x_2 = 0$, then it is clear that
$$
\lim_{\tau \to 0+} \Phi(\tau) \lim_{\varepsilon\to 0+} \Phi^{-1}(\varepsilon) \begin{pmatrix}
0\\
y_2\\
0\\
w_2
\end{pmatrix}  = 
\begin{pmatrix}
0\\
y_2\\
0\\
w_2
\end{pmatrix}.
$$
So it remains to find a single independent vector in $\cL_{0+}$ in this case. Let us slightly abuse the notation and denote
$$
\begin{pmatrix}
x_1(\varepsilon)\\
y_1(\varepsilon)\\
z_1(\varepsilon)\\
w_1(\varepsilon)
\end{pmatrix}
=
\Phi^{-1}(\varepsilon)
\begin{pmatrix}
x_1\\
y_1\\
1\\
w_1
\end{pmatrix}.
$$
Using an explicit expression for the fundamental matrix, we find that $x_1(\varepsilon)$ is a linear combination of the modified Bessel $K$-functions and $z_1(\varepsilon)$ is a linear combination of $I$-functions. Due to the exponential behaviour of $I_a(x)$ and $K_a(x)$, we find that
$$
\lim_{\varepsilon \to 0+} \frac{1}{z_1(\varepsilon)}\Phi^{-1}(\varepsilon)
\begin{pmatrix}
x_1\\
y_1\\
1\\
w_1
\end{pmatrix} = \begin{pmatrix}
0 \\
0\\
1\\
0
\end{pmatrix}.
$$
For the same reason
\begin{equation}
\label{eqsomeeq}
\lim_{\tau \to 0+} \frac{\sqrt{c_{11}}\tau^{\frac{1+\beta}{2}}}{\sqrt{b_{11}}K_{\beta^{-1}+1}(s(\tau))}\Phi(\tau)\begin{pmatrix}
0\\
0\\
1\\
0
\end{pmatrix} = \begin{pmatrix}
1\\
0\\
0\\
0
\end{pmatrix}
\end{equation}
and we obtain that $\cL_{0+} = \cL_0^{X(0)}$.

Assume now that $x_2\neq 0$. Then we obtain by the same argument as above 
$$
\lim_{\varepsilon \to 0+} \frac{2c_{11}}{\beta x_2\sqrt{\varepsilon}I_{\beta^{-1}}(s(\varepsilon))} \Phi^{-1}(\varepsilon)\begin{pmatrix}
x_2\\
y_2\\
0\\
w_2
\end{pmatrix} = \begin{pmatrix}
0\\
0\\
1\\
0
\end{pmatrix}.
$$ 
Exploiting once more formula \eqref{eqsomeeq}, we once again find that $X(0) \in \cL_{0+}$. To find an independent vector limit let us write down explicitly the vectors $n_1(\varepsilon)$ and $n_2(\varepsilon)$ defined in \eqref{eqn1n2} of the previous subsection. We have
$$
n_1(\varepsilon) = \begin{pmatrix}
x_1(\varepsilon)\\
z_1(\varepsilon)
\end{pmatrix} =
-\frac{2c_{11}}{\beta}\begin{pmatrix}
-\frac{\sqrt{b_{11}}\varepsilon^{-\frac{1+\beta}{2}}K_{\beta^{-1}+1}(s(\varepsilon))}{\sqrt{c_{11}}} + \sqrt{\varepsilon}x_1 K_{\beta^{-1}}(s(\varepsilon))\\
-\frac{\sqrt{b_{11}}\varepsilon^{-\frac{1+\beta}{2}}I_{\beta^{-1}+1}(s(\varepsilon))}{\sqrt{c_{11}}} - \sqrt{\varepsilon}x_1 I_{\beta^{-1}}(s(\varepsilon))
\end{pmatrix},
$$
$$
n_2(\varepsilon)  =
-\frac{2c_{11}}{\beta}\begin{pmatrix}
\sqrt{\varepsilon}x_2 K_{\beta^{-1}}(s(\varepsilon))\\
-\sqrt{\varepsilon}x_2 I_{\beta^{-1}}(s(\varepsilon)). 
\end{pmatrix}
$$
To find the independent limit vector above we consider
$$
\lambda(\varepsilon) = \lambda_2(\varepsilon) + \frac{\sqrt{\varepsilon}x_2 I_{\beta^{-1}}(s(\varepsilon))}{z_1(\varepsilon)}\lambda_1(\varepsilon).
$$
The only component of $\lambda_1(\varepsilon)$ and $\lambda_2(\varepsilon)$ escaping to infinity are the $z$-components as can be easily seen from the explicit expression of $n_i(\varepsilon)$. But the $z$-component of $\lambda(\varepsilon)$ is equal to zero. Moreover the coefficient in front of $\lambda_1(\varepsilon)$ tends to zero as $\varepsilon\to 0+$. Thus from the explicit expressions for $x_1(\varepsilon)$ and $x_2(\varepsilon)$ we obtain that
$$
\lim_{\varepsilon\to 0+} \lambda(\varepsilon) = \begin{pmatrix}
0 \\
y_2 \\
0\\
z_2
\end{pmatrix},
$$
which is a vector that does not lie in the singular invariant subspace. Thus as in the previous sections the same vector lies in $\cL_{0+}$ which proves the result.

\subsection{Jump for $m \geq2$}
\label{subsec:jump}
In the previous subsections we have seen, that for the autonomous models the Jacobi curve has the right limit $\cL_{0+} = \cL_0^{X(0)}$. Now we are ready to prove this for a general system \eqref{eqnormal_jacobi}. We use the standard Riccati comparison result from~\cite{reid}.
\begin{lemma}
\label{lemm:comparison}
Suppose that $B(\tau)$ and $C(\tau)$ are two symmetric continuous matrix functions that satisfy $B(\tau)\geq 0$ and $C(\tau)\geq 0$ for almost every $\tau$ of any closed subinterval $[a,b]$ of a given open interval $I$. Then given a symmetric matrix $S_a \geq 0$, any Cauchy solution of
\begin{align}
\dot{S} + SA+A^TS +SBS - C =0, \\
S(a) = S_a,\nonumber
\end{align}
satisfies $S(\tau) \geq 0$ for all $\tau\in[a,b]$. 
\end{lemma}

We consider now the general system \eqref{eqnormal_jacobi}. Let $q=Sp$ and we write the corresponding Riccati equation like discussed in Section~\ref{sec:sympl}
\begin{equation}
\label{eqriccati}
\dot{S} + \frac{SB(\tau)S}{\tau^m} - C(\tau) =0.
\end{equation}
If the system is not oscillating, then we have existence of the Cauchy problem with the boundary data $S(t) = S$ on the interval $(0,t]$ for any fixed symmetric matrix $S$ and for $t$ small enough. 

Assume that $B_1(\tau) \leq B(\tau) \leq B_2(\tau) \leq  0$ and $C_1(\tau) \geq C(\tau) \geq C_2(\tau)$ for small $\tau \in [0,t]$. We assume that $B_i$ and $C_i$ are diagonal matrices like in our models from the previous subsection satisfying the non-oscillation conditions. Then we can define $S^\varepsilon_i$ to be solutions of the Cauchy problem 
$$
\dot{S}+\frac{SB_i(\tau)S}{\tau^m} - C_i(\tau) =0, \qquad S(\varepsilon) = S_{\cL_0},
$$
where $S_{\cL_0}$ is a symmetric matrix that corresponds to $\cL_0$ assuming of course that $\cL_0$ is transversal to the horizontal plane $q=0$. Let $S^\varepsilon(\tau)$ be a solution of \eqref{eqriccati} with $S(\varepsilon) = S_{\cL_0}$. 

Let us assume, for example $W^\varepsilon(\tau) = S_1^\varepsilon(\tau) - S_2^\varepsilon(\tau)$. Then we have that $W^\varepsilon(\tau)$ satisfies
$$
\dot{W}^\varepsilon+ W^\varepsilon\frac{B_1(\tau)}{\tau^m}S_2^\varepsilon+ S_2^\varepsilon\frac{B_1(\tau)}{\tau^m}W^\varepsilon + W^\varepsilon\frac{B_1(\tau)}{\tau^m}W^\varepsilon + \frac{S_2^\varepsilon(B_1(\tau)-B_2(\tau))S_2^\varepsilon}{\tau^m} - (C_1(\tau) - C_2(\tau))=0
$$
with $W_\varepsilon(\varepsilon) = 0$. But then by the Lemma~\ref{lemm:comparison} we obtain that
$$
W^\varepsilon(\tau) \geq 0 \iff S_1^\varepsilon(\tau) \geq S_2^\varepsilon(\tau), 
$$
for any $\tau\geq\varepsilon$ as long as $S^\varepsilon_2(\tau)$ is defined.

By replacing $B_1(\tau)$ with $B(\tau)$ and then $B_2(\tau)$ with $B(\tau)$, we similarly obtain that
$$
S_2^\varepsilon(\tau) \leq S^\varepsilon(\tau) \leq S_1^\varepsilon(\tau),
$$
for any $\tau \geq \varepsilon$ sufficiently close to $\varepsilon$ and $\varepsilon>0$ small. By fixing $\tau$ sufficiently small and taking limits as $\varepsilon\to 0+$ we find that
$$
S_1(\tau) \leq S(\tau) \leq S_2(\tau),
$$
where these matrix functions are the corresponding Jacobi curves. But we have proven in the previous subsections that for our model examples we had the same right limit. Thus $S_1(0+)=S_2(0+)$ and 
$$
S(0+) = S_1(0+) = S_2(0+). 
$$

If $\cL_0$ or $\cL_0^{X(0)}$ are not transversal to the horizontal space $\Sigma$, then this construction clearly does not work, because either $S^\varepsilon(\varepsilon)$ or $S(0+)$ do not exist. In this case we make a change of variables of the form
$$
\begin{pmatrix}
p\\
q
\end{pmatrix}
\mapsto
M
\begin{pmatrix}
p\\
q
\end{pmatrix}, \qquad M = 
\begin{pmatrix}
\alpha_1 & 0 & \beta_1 & 0 \\
0 & \alpha_2 & 0 & \beta_2\\
\gamma_1 & 0 & \delta_1 & 0 \\
0 & \gamma_2 & 0 & \delta_2
\end{pmatrix},
$$
s.t.
$$
\begin{vmatrix}
\alpha_i & \beta_i \\
\gamma_i & \delta_i
\end{vmatrix} = 1
$$
Matrix $M$ is clearly symplectic and we want to choose it in such a way that $M\cL_0$ and $M\cL_0^{X(0)}$ are transversal to the horizontal subspace $\Sigma$. Such matrices $M$ are actually dense in the set of all matrices of the given form. We can prove this by an explicit computation. 

If $\dim(\cL_0 \cap \Sigma ) > 0$, then following along the lines of Example~\ref{ex:1} we can assume that
$$
\cL_0 =\begin{bmatrix}
x & 0\\
0 & 0\\
z & 0\\
0 & 1
\end{bmatrix}
$$
Then $\dim(M\cL_0 \cap \Sigma) = 0$ is equivalent to 
$$
\alpha_1 x + \beta_1 z \neq 0, \qquad 
\beta_2 \neq 0.
$$
Similarly from examples~\ref{ex:1} and~\ref{ex2} we know that if $\dim(\cL_0^{X(0)} \cap \Sigma) > 0$, then
$$
\cL_0^{X(0)} = \begin{bmatrix}
1 & 0 \\
0 & 0\\
0 & 0 \\
0 & 1
\end{bmatrix}
$$
Then $\dim(M\cL_0^{X(0)} \cap \Sigma ) = 0$ can be achieved by taking $\beta_2 \neq 0$.

In the new coordinates our Jacobi equation takes the form
$$
\frac{d}{d\tau}\begin{pmatrix}
p\\
q
\end{pmatrix} = M\begin{pmatrix}
0 & \frac{B(\tau)}{\tau^{m}}\\
C(\tau) & 0
\end{pmatrix}M^{-1}
\begin{pmatrix}
p\\
q
\end{pmatrix}.
$$ 
An explicit computation gives us
\begin{align*}
& M\begin{pmatrix}
0 & \frac{B(\tau)}{\tau^{m}}\\
C(\tau) & 0
\end{pmatrix}M^{-1} =  \\
= &
\begin{pmatrix}
c_{11} \beta_1 \delta_1 - \left(\frac{1}{b} +b_{11}\right)  \gamma_1 \alpha_1 & 
  -b_{12} \gamma_2 \alpha_1 & 
  -c_{11} \beta_1^2 + \left(\frac{1}{b} +b_{11}\right)\alpha_1^2 & 
  b_{12} \alpha_1 \alpha_2 \\
 -b_{12} \gamma_1 \alpha_2 & 
  c_{22} \delta_2 \beta_2 - b_{22} \gamma_2 \alpha_2 & 
  b_{12} \alpha_1 \alpha_2 & -c_{22} \beta_2^2 + b_{22} \alpha_2^2\\
 c_{11} \delta_1^2 - \left(\frac{1}{b} +b_{11}\right)  \gamma_1^2 & 
  -b_{12} \gamma_1 \gamma_2 & -c_{11} \delta_1 \beta_1 + 
  \left(\frac{1}{b} +b_{11}\right)  \gamma_1 \alpha_1 & b_{12} \gamma_1 \alpha_2 \\
 -b_{12} \gamma_1 \gamma_2 & c_{22} \delta_2^2 - b_{22} \gamma_2^2 & 
  b_{12} \gamma_2 \alpha_1 & 
  -c_{22} \delta_2 \beta_2 + b_{22} \alpha_2 \gamma_2
\end{pmatrix}
\end{align*}
Recall that our original system was such that $b<0$ and $b_{22} <0$ for $\tau\geq 0$ small. Thus the upper off-diagonal 2x2 block will be a negative matrix function for small $\tau>0$, if we choose $\alpha_2$ big enough. For the same reason the lower off-diagonal 2x2 block will be negative if we choose $\gamma_2$ big enough. Thus we can apply the comparison lemma as before with
\begin{align*}
B_1(\tau) &= \begin{pmatrix}
\frac{\alpha_1^2}{b(\tau)}+\varepsilon & 0\\ 
0 &  - c_{22}(0)\beta_2^2 + b_{22}(0)\alpha_2^2 +\varepsilon
\end{pmatrix}, 
\\
B_2(\tau) &= \begin{pmatrix}\frac{\alpha_1^2}{b(\tau)}-\varepsilon & 0\\ 
0 & - c_{22}(0)\beta_2^2 + b_{22}(0)\alpha_2^2 -\varepsilon
\end{pmatrix},\\
C_1(\tau) &= \begin{pmatrix}
-\frac{\gamma_1^2}{b(\tau)}-\varepsilon & 0\\ 
0 & c_{22}(0)\delta_2^2-b_{22}(0)\gamma_2^2-\varepsilon
\end{pmatrix}, 
\\
C_2(\tau) &= \begin{pmatrix}
-\frac{\gamma_1^2}{b(\tau)}+\varepsilon & 0\\ 
0 & c_{22}(0)\delta_2^2-b_{22}(0)\gamma_2^2 +\varepsilon
\end{pmatrix}
\end{align*}
where $\varepsilon >0$ is sufficiently small.

\subsection{Jacobi curve for $m=1,2$}
\label{subsec:curve}
As we have already discussed before, the jump alone does not determine the Jacobi curve, because with a singular Jacobi DE we lose uniqueness. So we need to characterize the right solution of the Jacobi equation. In this section we prove the following result.
\begin{theorem}
\label{thm:first_jet}
If $m=1$ or $m=2$ then Jacobi curve after a singularity can be characterized as a boundary value problem of the extended Jacobi DE on the Lagrangian Grassmanian with conditions on the left end-point and the first left derivative.
\end{theorem}

This will be proven in a number of steps:
\begin{enumerate}
\item We change coordinates so that $\cL_0$ and $\cL_{0+}$ lie in the same coordinate chart and $\cL_{0+}$ is taken to be zero;
\item We write down the corresponding Riccati equation and perform a certain blow-up procedure;
\item After the blow-up we obtain a non-autonomous Riccati equation. We then proceed in determining the Jacobi curve for the autonomous part;
\item Using a deformation argument we prove that in the non-autonomous case the Jacobi curve is well-defined by the same jet.
\end{enumerate}

For the first step we are going to have three different situations as well
\begin{enumerate}
\item $\cL_0$ is transversal to the horizontal plane in current coordinates and in the corresponding symmetric matrix $S^0$ either $S_{11}^0\neq 0$ or $S_{11}^0 = S_{12}^0 = 0$;
\item $\cL_0$ is either transversal to the horizontal plane in current coordinates and in the corresponding symmetric matrix $S_{11}^0= 0$, $S_{22}^0 \neq 0$ or $\cL_0$ and the horizontal plane $\Sigma$ have a common line;
\item $\cL_0$ is either transversal to the horizontal plane in current coordinates and in the corresponding symmetric matrix $S_{11}^0= S_{22}^0= 0$, $S_{12}^0 \neq 0$ or $\cL_0$ is the horizontal plane $\Sigma$.
\end{enumerate}

Let
$$
S_{22}^+ = \begin{cases}
S_{22}^0, & S_{11}^0=0 \\
S_{22}^0-\frac{(S_{12}^0)^2}{S_{11}^0}, & S_{11}^0\neq 0
\end{cases}
$$
If $S_{11}^0 \neq 0$ or if $S_{11}^0 = S_{12}^0 = 0$, then we have 
$$
\cL_{0+} = \begin{bmatrix}
1 & 0 \\
0 & 1 \\
0 & 0\\
0& S_{22}^+
\end{bmatrix}
$$
or else, when $S_{11}^0 = 0$ and $S_{12}^0 \neq 0$, we get
$$
\cL_{0+} = \begin{bmatrix}
1 & 0 \\
0 & 0 \\
0 & 0\\
0 & 1
\end{bmatrix}
$$
Depending on the case we apply one of the three symplectic transformations
$$
M_1 = \begin{pmatrix}
1 & 0 & 0 & 0 \\
0 & 1 & 0 & 0 \\
0 & 0 & 1 & 0 \\
0 & -S_{22}^+ & 0 & 1 
\end{pmatrix},
\qquad
M_2 = \begin{pmatrix}
1 & 0 & 0 & 0 \\
0 & 0 & 0 & -1 \\
0 & 0 & 1 & 0 \\
0 & 1 & 0 & 0 
\end{pmatrix},
\qquad
M_3 = \begin{pmatrix}
1 & 0 & -1 & 0 \\
0 & 0 & 0 & -1 \\
0 & 0 & 1 & 0 \\
0 & 1 & 0 & 0 
\end{pmatrix}.
$$
These transformations map $\cL_{0+}$ to the vertical subspace. Let us check what happens to $\cL_0$ under these transformations. We have for case 1 either
$$
M_1\begin{bmatrix}
1 & 0 \\
0 & 1 \\
S_{11}^0 & S_{12}^0\\
S_{12}^0 & S_{22}^0
\end{bmatrix} =
\begin{bmatrix}
1 & 0 \\
0 & 1 \\
S_{11}^0 & S_{12}^0\\
S_{12}^0 & \frac{(S_{12}^0)^2}{S_{11}^0}
\end{bmatrix} 
\qquad \text{or} \qquad 
M_1\begin{bmatrix}
1 & 0 \\
0 & 1 \\
0 & 0\\
0 & S_{22}^0
\end{bmatrix} =
\begin{bmatrix}
1 & 0 \\
0 & 1 \\
0 & 0\\
0 & 0
\end{bmatrix}.
$$

For case 2 either
$$
M_2\begin{bmatrix}
1 & 0 \\
0 & 1 \\
0 & S_{12}^0\\
S_{12}^0 & S_{22}^0
\end{bmatrix}=
\begin{bmatrix}
1 & 0 \\
-S_{12}^0 & -S_{22}^0 \\
0 & S_{12}^0\\
0 & 1
\end{bmatrix} =
\begin{bmatrix}
1 & 0\\
0 & 1\\
-\frac{(S_{12}^0)^2}{S_{22}^0} & -\frac{(S_{12}^0)}{S_{22}^0}\\
-\frac{(S_{12}^0)}{S_{22}^0} & -\frac{1}{S_{22}^0}
\end{bmatrix}
$$
or
$$
M_2\begin{bmatrix}
1 & 0 \\
0 & 0 \\
z & 0\\
0 & 1
\end{bmatrix} =
\begin{bmatrix}
1 & 0 \\
0 & -1 \\
z & 0\\
0 & 0
\end{bmatrix} =
\begin{bmatrix}
1 & 0 \\
0 & 1 \\
z & 0\\
0 & 0
\end{bmatrix}.
$$

For the case 3 either 
$$
M_3\begin{bmatrix}
1 & 0 \\
0 & 1 \\
0 & S_{12}^0\\
S_{12}^0 & 0
\end{bmatrix} =
\begin{bmatrix}
1 & -S_{12}^0\\
-S_{12}^0 & 0\\
0 & S_{12}^0\\
0 & 1
\end{bmatrix} =
\begin{bmatrix}
1 & 0\\
0 & 1\\
-1 & -\frac{1}{S_{12}^0}\\
-\frac{1}{S_{12}^0} & -\frac{1}{(S_{12}^0)^2}
\end{bmatrix}
$$
or
$$
M_3\begin{bmatrix}
0 & 0 \\
0 & 0 \\
-1 & 0\\
0 & -1
\end{bmatrix} =
\begin{bmatrix}
1 & 0 \\
0 & 1 \\
-1 & 0\\
0 & 0
\end{bmatrix}.
$$
And this finishes the first step.

For the second step we have to rewrite the Jacobi equation in the new coordinates. We simply have to conjugate the right-hand side of \eqref{eqnormal_jacobi} by the corresponding matrix $M_i$. Then to each case corresponds its own Jacobi equation of the form 
$$
\frac{d}{d\tau}\begin{pmatrix}
p\\
q
\end{pmatrix}=
M_i\begin{pmatrix}
A(\tau) & B(\tau)\\
C(\tau) & -A^T(\tau)
\end{pmatrix}M_i^{-1}
\begin{pmatrix}
p\\
q
\end{pmatrix}
$$
or more precisely
\begin{align*}
&\text{Case 1}: &\frac{d}{d\tau}\begin{pmatrix}
p_1\\
p_2\\
q_1\\
q_2
\end{pmatrix} &= 
\begin{pmatrix}
0 & b_{12}(\tau)S_{22}^+ & \frac{1}{b(\tau)} + b_{11}(\tau) & b_{12}(\tau)\\
0 & b_{22}(\tau)S_{22}^+ & b_{12}(\tau) & b_{22}(\tau)\\
c_{11}(\tau) & 0 & 0 & 0\\
0 & c_{22}(\tau)-b_{22}(\tau)(S_{22}^+)^2 & -b_{12}(\tau)S_{22}^+ & -b_{22}(\tau)S_{22}^+
\end{pmatrix}
\begin{pmatrix}
p_1\\
p_2\\
q_1\\
q_2
\end{pmatrix},\\
&\text{Case 2}: &\frac{d}{d\tau}\begin{pmatrix}
p_1\\
p_2\\
q_1\\
q_2
\end{pmatrix} &= 
\begin{pmatrix}
0 & -b_{12}(\tau) & \frac{1}{b(\tau)} + b_{11}(\tau) & 0\\
0 & 0 & 0 & -c_{22}(\tau)\\
c_{11}(\tau) & 0 & 0 & 0\\
0 & -b_{22}(\tau) & b_{12}(\tau) & 0
\end{pmatrix}
\begin{pmatrix}
p_1\\
p_2\\
q_1\\
q_2
\end{pmatrix},\\
&\text{Case 3}: &\frac{d}{d\tau}\begin{pmatrix}
p_1\\
p_2\\
q_1\\
q_2
\end{pmatrix} &= 
\begin{pmatrix}
-c_{11}(\tau) & -b_{12}(\tau) & \frac{1}{b(\tau)} + b_{11}(\tau)-c_{11}(\tau) & 0\\
0 & 0 & 0 & -c_{22}(\tau)\\
c_{11}(\tau) & 0 & c_{11}(\tau) & 0\\
0 & -b_{22}(\tau) & b_{12}(\tau) & 0
\end{pmatrix}
\begin{pmatrix}
p_1\\
p_2\\
q_1\\
q_2
\end{pmatrix}.
\end{align*}

We then take $q=Sp$ and obtain a Riccati equation of the form \eqref{eqriccati_general}. We do a blow-up of this equation by taking
$$
S(\tau) = \tau S_1(\tau).
$$
Then we obtain a Riccati equation for $S_1(\tau)$ of the form
\begin{equation}
\label{eqriccati_deriv2}
\tau \dot{S}_1 +S_1 + S_1 \begin{pmatrix}
\frac{1}{b_2} & 0\\
0 & 0
\end{pmatrix} S_1 - 
C(0) = \tau R(\tau,S)
\end{equation}
for $m = 2$ and
\begin{equation}
\label{eqriccati_deriv1}
\tau \dot{S}_1 +S_1  - 
C(0) = \tau R(\tau,S)
\end{equation}
for $m = 1$. 

Let $\tilde{S}^0$ be the symmetric matrix that corresponds to $\cL_0$ in the new coordinates. We denote by $S^\varepsilon_1(\tau)$ the solution of this Riccati equation which satisfies
$$
S_1^\varepsilon(\varepsilon) = \frac{\tilde{S}^0 }{\varepsilon}.
$$
Since outside the singularity the right-hand side is analytic and we have a family of solutions converging to a solution, it is clear that for the limiting curve
$$
\dot{S}(\tau) = \lim_{\varepsilon \to 0+} \left( S^\varepsilon_1(\tau) + \tau \dot{S}_1^\varepsilon(\tau)\right),
$$
for $\tau >0$ sufficiently small. That finishes the second step. 

For the third step we are going to consider just the first case. For the second and the third case the argument is repeated word by word (see Remark~\ref{rmk:new}). We assume that the right-hand side of those equations is actually zero. Then we can understand very well the whole phase portrait of this Riccati equation. Indeed, we can extend the dynamics to the whole Lagrangian Grassmanian $L(2)$ by rewriting the corresponding Hamiltonian system. 
$$
\tau\frac{d}{d\tau}\begin{pmatrix}
p_1\\
p_2\\
q_1\\
q_2
\end{pmatrix} = 
\begin{pmatrix}
\frac{1}{2} & 0 & \frac{1}{b_2} & 0 \\
0 & \frac{1}{2} &0 & 0 \\
c_{11}(0) & 0 & -\frac{1}{2} & 0 \\
0 & c_{22}(0)-(S_{22}^+)^2b_{22}(0) & 0 &  -\frac{1}{2} 
\end{pmatrix}\begin{pmatrix}
p_1\\
p_2\\
q_1\\
q_2
\end{pmatrix} = H\begin{pmatrix}
p_1\\
p_2\\
q_1\\
q_2
\end{pmatrix}
$$

A complete description of the phase portrait of such a system was given in~\cite{riccati}. It is clear that the equilibrium points are spanned by  the eigenvectors. In our case, $H$ has eigenvalues
$$
\lambda_1 = -\frac{1}{2}\sqrt{1+\frac{4c_{11}(0)}{b_2}},\quad \lambda_2 = -\frac{1}{2}, \quad
\lambda_3 = \frac{1}{2}, \quad 
\lambda_4 = \frac{1}{2}\sqrt{1+\frac{4c_{11}(0)}{b_2}}.
$$
From the assumptions we have that all four eigenvalues are real and different. Let $E_i$ be the corresponding eigenvectors. We have
\begin{align*}
E_1 &= \begin{pmatrix}
 1-\sqrt{1+\frac{4c_{11}(0)}{b_2}} \\
0\\
2c_{11}(0)\\
0
\end{pmatrix},  &
E_2 &= \begin{pmatrix}
0\\
0\\
0\\
1
\end{pmatrix},\\
E_3 &= \begin{pmatrix}
0\\
1\\
0\\
c_{22}(0)-b_{22}(0)(S_{22}^+)^2
\end{pmatrix}, 
&
E_4 &= \begin{pmatrix}
 1+\sqrt{1+\frac{4c_{11}(0)}{b_2}} \\
0\\
2c_{11}(0)\\
0
\end{pmatrix}.
\end{align*}

We define $E_{ij} = \spn\{E_i,E_j\}$. It is easy to see that we have four equilibrium points on the Lagrangian Grassmanian: $E_{12}, E_{13}, E_{24}, E_{34}$. For each of these equilibrium points we can find the corresponding stable and unstable manifolds $W^s(E_{ij})$ and $W^u(E_{ij})$. Then if $S_{1}^\varepsilon(\varepsilon)$ lies in $W^s(E_{ij})$, the Jacobi curve is going to be just the equilibrium solution $\cL_\tau = E_{ij}$. Indeed, the Lagrangian Grassmanian is compact and therefore any trajectory in the stable manifold has finite length. But every trajectory of our Riccati equation has speed that goes to infinity as $\tau \to 0+$. So as we take $\varepsilon$ smaller and smaller for a fixed time $\tau>0$ the corresponding curve $S^\varepsilon(\tau)$ is going to get closer and closer to the equilibrium point approaching it in the limit. It remains only to describe stable manifolds of our equilibrium points.

Luckily it was already done in~\cite{riccati} by M. Shayman. He proved that in order to find the stable manifolds we need to form a flag $\{0\} = V_0\subset V_1 \subset ... \subset V_4 = \R^4$, where 
$$
V_i = \bigoplus_{j = 1}^i E_j,
$$ 
and associate to each $E_{ij}$ a sequence $l(E_{ij})=(l_1,l_2,l_3,l_4)$, where 
$$
l_k = \begin{cases}
1 & \text{ if } k=i,j;\\ 
0 & \text{ otherwise. }  
\end{cases}.
$$
Then
$$
W^s(E_{ij}) = \left\{\Lambda \in L(2)\,:\, \dim \Lambda\cap V_m = \sum_{k=1}^m l_k,\, l_k \in l(E_{ij}),m=1,2,3,4.\right\}.
$$

It remains to check for which initial data $S_1^\varepsilon(\varepsilon)$ lies in which $W^s(E_{ij})$ for small $\varepsilon>0$ and describe the corresponding $W^s(E_{ij})$.
\begin{lemma}
Suppose that the right hand side of (\ref{eqriccati_deriv2}) is zero. Then the curves $\Lambda_1^\varepsilon(\tau)$ that correspond to $S^{\varepsilon}_1(\tau)$ converge pointwise to the equilibrium solution $\Lambda(\tau)\equiv E_{34}$. Or in local coordinates we get that
$$
\lim_{\varepsilon \to 0+}S_1^\varepsilon(\tau) = \begin{pmatrix}
-\frac{b_2}{2}\left( 1-\sqrt{1 + \frac{4c_{11}}{b_2}} \right) & 0\\
0 & c_{22}(0)-b_{22}(0)(S_{22}^+)^2 
\end{pmatrix}= S^{34}_1.
$$
If $R(\tau,S)=0$ in (\ref{eqriccati_deriv2}), then exists a unique solution of this equation with $S_1(0+) = S^{34}_1$.
\end{lemma}

\begin{proof}
By definition we find that
$$
W^s(E_{34}) = \left\{\Lambda\in L(2) \,:\, \dim(\Lambda \cap E_{12}) = 0\right\} = E_{12}^\pitchfork,
$$
which is dense in $L(2)$. So we only need to prove that $\Lambda^\varepsilon(\varepsilon)\in E_{12}^\pitchfork$ for $\varepsilon>0$ small. Indeed, in this case the unstable manifold $W^u(E_{34}) = \{E_{34}\}$ and so the only solution of (\ref{eqriccati_deriv2}) with $S_1(0+) = S^{34}_1$ can be $S_1(\tau) \equiv S_1^{34}$.

We note that if $S_{11}^0 = S_{12}^0 = 0$, then $S^\varepsilon_1(\varepsilon) = 0$. In this case for small $\varepsilon>0$ it is clear that $\dim (\Lambda^{\varepsilon}(\varepsilon)\cap E_{12})=0$. If  $S_1^\varepsilon(\varepsilon) \neq 0$ for small $\varepsilon>0$, then $\dim (\Lambda^{\varepsilon}(\varepsilon)\cap E_{12})>0$ if and only if
$$
S_{12}^+ = 0 \quad \text{   and   } \quad \frac{S_{11}^+}{\varepsilon} = - \frac{b_2}{2}\left( 1+\sqrt{1+\frac{4c_{11}}{b_2}} \right),
$$
but this can happen only for a single value
$$
\varepsilon = -\frac{2S_{11}^+}{b_2\left( 1+\sqrt{1+\frac{4c_{11}}{b_2}}\right)}.
$$
And so for small $\varepsilon>0$ we indeed get $\Lambda^\varepsilon(\varepsilon)\in E_{12}^\pitchfork$.
\end{proof}

Case $m = 1$ is easier, since the principal part of the equation (\ref{eqriccati_deriv1}) is linear and has a global stable equilibrium
$$
S = \begin{pmatrix}
c_{11}(0) & 0\\
0 & c_{22}(0)-b_{22}(0)(S_{22}^+)^2
\end{pmatrix}.
$$
As for $m = 2$, we have then $S_1(0+) = S$ and a unique solution to a Cauchy problem, that characterizes our Jacobi curve.

\begin{remark}
\label{rmk:new}
For the case 2 and 3 we have a similar result. We obtain that 
$$
\lim_{\varepsilon \to 0+}S_1^\varepsilon(\tau) = \begin{pmatrix}
-\frac{b_2}{2}\left( 1-\sqrt{1 + \frac{4c_{11}}{b_2}} \right) & 0\\
0 & -b_{22}(0) 
\end{pmatrix} = S_1^{34}
$$
for $m =2 $ and
$$
\lim_{\varepsilon \to 0+}S_1^\varepsilon(\tau) =
\begin{pmatrix}
c_{11}(0) & 0\\
0 & -b_{22}(0)
\end{pmatrix}
$$
for $m = 1$ and that in this case indeed the Jacobi curve is fully determined by the first jet. We keep the notation $S_1^{34}$ because in the case 2 and 3 we obtain a Hamiltonian system whose matrix has exactly the same eigenvalues as the Hamiltonian matrix of case 1 and the same eigenvectors except $E_3$ that must replaced by
$$
E_3 = \begin{pmatrix}
0\\
1\\
0\\
-b_{22}(0)
\end{pmatrix}
$$
\end{remark}

It remains now to do the last step and to show the general case. Let us assume 
$$
S_1 = \begin{pmatrix}
S_{11} & S_{12}\\
S_{12} & S_{22}
\end{pmatrix}
$$
and rewrite \eqref{eqriccati_deriv2} or \eqref{eqriccati_deriv1} as a system on $\R\times L(2)$. Namely we have
\begin{align*}
\dot{S} &= Q(S) + \tau R(\tau,S),\\
\dot{\tau} &= \tau;
\end{align*}
where $Q(S)$ is the autonomous Riccati part. It is clear that $(S_{1}^{34},0)$ is an equilibrium point of this system. Moreover, by linearising the right hand side at $(S_{1}^{34},0)$ we obtain that it is a hyperbolic equilibrium point, since the linearized operator has eigenvalues
$$
\left\{ -\sqrt{1+\frac{4 c_{11}}{b_2}},-\frac{1}{2}-\frac{1}{2}\sqrt{1+\frac{4 c_{11}}{b_2}},-1,1 \right\},
$$
the same as for the autonomous system in all three cases. So by Grobman-Hartman theorem both systems are topologically conjugate in the neighbourhood of this equilibrium point. Since both of them have a single unstable direction it means that there exists a unique trajectory of the non-autonomous system that approaches $(S_{1}^{34},0)$ as $\tau \to 0+$. We claim that this trajectory must be a lift of the Jacobi curve to the extended phase-space. This result does not follow directly from the Grobman-Hartman theorem since $S^\varepsilon(\varepsilon)$ is far from the equilibrium and a priori we have no information about the behaviour orbits close to infinity. 

The result follows from an application of the variation formulae proved in~\cite{as} that can be stated as follows. Given a non-autonomous vector field $Z_s$ we denote by $F_t[Z_{\cdot}]$ a flow from time $0$ to time $t$ of the corresponding vector field. If $X,Y$ is a pair of autonomous vector fields, the variation formulae reads as
$$
F_t[X+Y] = F_t[(F_{.-t}[X])_* Y]\circ F_t[X].
$$  
where $F_{s-t}[X]_* Y$ is just the push-forward of $Y$ under $F_{s-t}[X]$.

In our case 
$$
X= \begin{pmatrix}
Q(S)\\
\tau
\end{pmatrix}, \qquad
Y = \begin{pmatrix}
\tau R(\tau,S)\\
0
\end{pmatrix}. 
$$
Due to smoothness of each flow the lift of the Jacobi curve will be then given by the limit curve
\begin{align*}
&\lim_{\varepsilon \to 0+} F_{\tau-\varepsilon}[X+Y]\begin{pmatrix}
S_1^\varepsilon(\varepsilon)\\
\varepsilon
\end{pmatrix}
= \lim_{\varepsilon \to 0+} \left( F_{\tau-\varepsilon}(F_{.-\tau+\varepsilon}[X])_* Y]\circ F_{\tau-\varepsilon}[X]\begin{pmatrix}
S_1^\varepsilon(\varepsilon)\\
\varepsilon
\end{pmatrix} \right) =\\
&=\lim_{\varepsilon \to 0+} \left( F_{\tau-\varepsilon}(F_{.-\tau+\varepsilon}[X])_* Y]\right)\circ \lim_{\varepsilon \to 0+} \left(F_{\tau-\varepsilon}[X]\begin{pmatrix}
S_1^\varepsilon(\varepsilon)\\
\varepsilon
\end{pmatrix} \right) =  \\
&=F_{\tau}(F_{.-\tau}[X])_* Y]\circ \lim_{\varepsilon \to 0+} \left(F_{\tau-\varepsilon}[X]\begin{pmatrix}
S_1^\varepsilon(\varepsilon)\\
\varepsilon
\end{pmatrix} \right)
\end{align*}
But the second limit corresponds to the lift of the Jacobi curve in autonomous case. Thus if we take a limit of this expression as $\tau \to 0+$ we would obtain
$$
\begin{pmatrix}
S_1^{34}\\
0
\end{pmatrix},
$$
like in the autonomous case. Which proves Theorem~\ref{thm:first_jet}.

\bibliographystyle{plain}
\bibliography{references}

\begin{thebibliography}{10}

\bibitem{agr_lderiv}
A.~Agrachev.
\newblock {Feedback--invariant optimal control theory and differential
  geometry, II. Jacobi curves for singular extremals}.
\newblock {\em J. Dynamical and Control Systems}, pages 583--604, 1998.

\bibitem{agrachev_cime}
A.~Agrachev.
\newblock {\em {Nonlinear and Optimal Control Theory}}, chapter Geometry of
  Optimal Control Problems and Hamiltonian Systems, pages 1--59.
\newblock Springer, 2004.

\bibitem{A_and_me}
A.~Agrachev and I.~Beschastnyi.
\newblock {Jacobi Fields in Optimal control: Morse and Maslov Indices}.
\newblock preprint.

\bibitem{agrachev_moi}
A.~Agrachev and I.~Beschastnyi.
\newblock {Symplectic geometry of constrained optimization}.
\newblock {\em Regular and Chaotic Dynamics}, 22:750--770, 2017.

\bibitem{agrachev_bang2}
A.~Agrachev and R.~Gamkrelidze.
\newblock {\em {Nonlinear Controllability and Optimal Control}}, chapter
  Symplectic geometry for optimal control, pages 263--277.
\newblock CRC Press Book, 1990.

\bibitem{agr_lderiv_1}
A.~Agrachev and R.~Gamkrelidze.
\newblock {Feedback--invariant optimal control theory and differential
  geometry, I. Regular extremals}.
\newblock {\em J. Dynamical and Control Systems}, pages 343--389, 1997.

\bibitem{agr_gamk_symp}
A.~Agrachev and R.~Gamkrelidze.
\newblock {\em {Geometry of Feedback and Optimal Control}}, chapter Symplectic
  methods in optimization and control, pages 19--77.
\newblock CRC Press Book, 1998.

\bibitem{luca}
A.~Agrachev, L.~Rizzi, and P.~Silveira.
\newblock {On conjugate times of LQ optimal control problems}.
\newblock {\em J. Dynamical and Control Systems}, 21:625--641, 2015.

\bibitem{as}
A.~Agrachev and Yu. Sachkov.
\newblock {\em {Control theory from a geometric point of view}}.
\newblock Springer, 2004.

\bibitem{agrachev_bang}
A.~Agrachev, G.~Steffani, and P.~Zezza.
\newblock {Strong optimality of a bang-bang trajectory}.
\newblock {\em SIAM J. on Control and Optimization}, 41:981--1014, 2002.

\bibitem{arnold2}
V.~Arnold.
\newblock {The Sturm theorems and symplectic geometry}.
\newblock {\em Funct. Anal. Appl.}, 19:251--259, 1985.

\bibitem{bonnard_book}
B.~Bonnard and M.~Chyba.
\newblock {\em {Singular Trajectories and Their Role in Control Theory}}.
\newblock Springer, 2003.

\bibitem{caillau}
J.-B. Caillau, J.~Fejoz, M.~Orieux, and R.~Roussarie.
\newblock {Singularities of min time affine control systems}.
\newblock preprint, hal-01718345, version 1.

\bibitem{maslov}
S.~Cappell, R.~Lee, and E.~Miller.
\newblock {On the Maslov index}.
\newblock {\em Comm. Pure Appl. Math.}, 47:121--186, 1994.

\bibitem{ode}
E.~Coddington and N.~Levinson.
\newblock {\em {Theory of Ordinary Differential Equations}}.
\newblock Krieger Pub. Co., 1984.

\bibitem{gosson}
M.~de~Gosson.
\newblock {\em {Symplectic Geometry and Quantum Mechanics}}.
\newblock Birkhauser, 2000.

\bibitem{sternberg}
V.~Guillemin and S.~Sternberg.
\newblock {\em {Geometric Asymptotics}}.
\newblock American Mathematical Society, 1977.

\bibitem{horn}
R.~Horn and C.~Johnson.
\newblock {\em {Matrix Analysis}}.
\newblock Cambridge University Press, 2 edition, 2012.

\bibitem{morse3}
M.~Morse.
\newblock {Singular quadratic functionals}.
\newblock {\em Math. Ann.}, 201:60--76, 1973.

\bibitem{morse2}
M.~Morse and W.~Leighton.
\newblock {Singular quadratic functionals}.
\newblock {\em Trans. Amer. Math. Soc.}, 36:252--286, 1936.

\bibitem{bessel}
F.~Olver and L.~Maximon.
\newblock {\em {NIST Handbook of Mathematical Functions}}, chapter Bessel
  Functions, pages 215--286.
\newblock Cambridge University Press, 2010.

\bibitem{osmolovskii}
N.~Osmolovskii and H.~Maurer.
\newblock {\em {Advances in Mathematical Modeling, Optimization and Optimal
  Control}}, chapter Second-Order Optimality Conditions for Broken Extremals
  and Bang-Bang Controls, pages 147--201.
\newblock Springer, 2016.

\bibitem{reid}
W.~Reid.
\newblock {\em {Riccati Differential Equations}}.
\newblock Academic Press, 1972.

\bibitem{goc}
H.~Sh{\"a}ttler and U.~Ledzewicz.
\newblock {\em {Geometric Optimal Control}}.
\newblock Springer, 2012.

\bibitem{riccati}
M.~Shayman.
\newblock {Phase Portrait of the Matrix Riccati Equation}.
\newblock {\em SIAM J. Control and Optimization}, 24(1):1--65, 1986.

\bibitem{SussmannLiu}
H.~Sussmann and W.~Liu.
\newblock {Shortest paths for sub-Riemannian metrics on rank-two
  distributions}, 1995.

\bibitem{sussmann}
H.~J. Sussmann.
\newblock {\em {Algebraic and Geometric Methods in Nonlinear Control Theory}},
  chapter Envelopes, Conjugate Points, and Optimal Bang-Bang Extremals, pages
  325--346.
\newblock D. Reidel Publishing company, 1986.

\bibitem{ode2}
G.~Teschl.
\newblock {\em {Ordinary Differential Equations and Dynamical Systems}}.
\newblock American Mathematical Society, 2012.

\bibitem{vagner}
V.~Wagner.
\newblock {The geometrical theory of the simplest n-dimensional singular
  problem of the calculus of variations}.
\newblock {\em Rec. Math. [Mat. Sbornik] N.S.}, 63:321--364, 1947.

\end{thebibliography}

\end{document}